\long\def\skipit#1{} %can skip over \par

\newcommand{\mdef}[1]{\textit{\textbf{#1}}}
\newcommand{\noi}{\noindent}

\newcommand{\vsa}{\vskip-12pt}
\newcommand{\vsb}{\vskip-6pt}
\newcommand{\vsc}{\vskip-3pt}
\newcommand{\hsb}{\hskip-6pt}
\newcommand{\hsc}{\hskip-3pt}

\newcommand{\rmand}{\quad\hbox{ and }\quad} %DW
\def\fa#1{\qquad\text{for all $#1$}} %DW

\newcommand{\tLC}{\mathcal{L}^\sim}  %JG
\newcommand{\gtLC}{\mathcal{L}^\gs}  %JG
\newcommand{\ltLC}{\mathcal{L}^\ls}  %JG
\newcommand{\A}{\mathcal{A}}
\newcommand{\B}{\mathcal{B}}
\newcommand{\W}{\mathcal{W}}
\newcommand{\ls}{\lesssim}
\newcommand{\gs}{\gtrsim}

\documentclass[12pt]{amsart}

\advance\textheight by \topskip
\usepackage{amsfonts,amsmath,amssymb,enumerate,cite}
\usepackage{ifthen,calc}
\usepackage{epsfig,graphicx}
\usepackage{latexsym}
\usepackage{multicol}
\usepackage[format=hang, margin=10pt]{caption}

\graphicspath{{VA-figs}}

\newcounter{hours}
\newcounter{minutes}
\newcommand{\printtime}{
	\setcounter{hours}{\time/60}%
	\setcounter{minutes}{\time-\value{hours}*60}
	\ifthenelse{\value{hours}<10}{0}{}\thehours:%
	\ifthenelse{\value{minutes}<10}{0}{}\theminutes}

\numberwithin{equation}{section}
\numberwithin{figure}{section}
\numberwithin{table}{section}

\newtheorem{thm}{Theorem}[section]
\newtheorem{cor}[thm]{Corollary}
\newtheorem{lemma}[thm]{Lemma}
\newtheorem{prop}[thm]{Proposition}
\newtheorem{conj}[thm]{Conjecture}

\newtheorem{J-com}{JG-comment}[section]

\theoremstyle{definition}
\newtheorem{example}{Example}[section]

\usepackage{xcolor}

\def\DWc#1{{\textcolor{cyan}{#1}}}
\def\DWm#1{{\textcolor{magenta}{#1}}}

\begin{document}% End of preamble and beginning of text.
%%%%%%%%%%%%%%%%%%%%%%%%%%%%%%%%%%%%%%%%%%%%%

\title[Log-Concavity of Combinations of Sequences ]{{Log-Concavity of Combinations of Sequences and Applications to Genus Distributions}}

\author[J.L. Gross, T. Mansour, T.W. Tucker, and D.G.L. Wang]{Jonathan L. Gross
}
\address{
Department of Computer Science  \\
Columbia University, New York, NY 10027, USA; \quad
email: gross@cs.columbia.edu
}
\author[]{Toufik Mansour
}
\address{
Department of Mathematics  \\
University of Haifa, 31905 Haifa, Israel;  \quad
email: tmansour@univ.haifa.ac.il}
\author[]{Thomas W. Tucker
}
\address{
Department of Mathematics  \\
Colgate University, Hamilton, NY 13346, USA; \quad
email: ttucker@colgate.edu
}
\author[]{David G.L. Wang
}
\address{
Department of Mathematics  \\
University of Haifa, 31905 Haifa, Israel;  \quad
email: david.combin@gmail.com; wgl@math.haifa.ac.il}

\date{}

\begin{abstract}  %\vskip-12pt
We formulate conditions on a set of log-concave sequences, under which any linear combination of those sequences is log-concave, and further, of conditions under which linear combinations of log-concave sequences that have been transformed by convolution are log-concave. These conditions involve relations on sequences called  \textit{synchronicity} and \textit{ratio-dominance}, and a characterization of some bivariate sequences as \textit{lexicographic}.   We are motivated by the 25-year old conjecture that the genus distribution of every graph is log-concave.  Although calculating genus distributions is NP-hard, they have been calculated explicitly for many graphs of tractable size, and the three conditions have been observed to occur in the \textit{partitioned genus distributions} of all such graphs.  They are used here to prove the log-concavity of the genus distributions of graphs constructed by iterative amalgamation of double-rooted graph fragments whose genus distributions adhere to these conditions, even though it is known that the genus polynomials of some such graphs have imaginary roots.  A blend of topological and combinatorial arguments demonstrates that log-concavity is preserved through the iterations.   \vskip-24pt
\end{abstract} \vskip-24pt

\maketitle     \vskip-24pt               % Produces the title.
\tableofcontents

\skipit{ %%
\begin{flushleft}   \vskip-24pt
\textbf{Version:\quad\printtime\quad\today}
\end{flushleft}
} %%

\enlargethispage{36pt}
\eject

%\markleft{\leftheaderfont{\textsc{Gross, Mansour, Tucker, and Wang}}}

%%%%%%%%%%%%%%%%%%%%%%%%%%%%%%%%%%%%%%%%%%%%
%%%%%%%%%%%%%%%%%%%%%%%%%%%%%%%%%%%%%%%%%%%%
\section{\large Introduction}\label{sec:intro}
\enlargethispage{-12pt}

The aim of this paper is two-fold.
We are motivated by a long-standing conjecture \cite{GrRoTu89}
that the genus distribution of every graph is log-concave.
Our initial objective was to confirm this conjecture for several families of graphs.
We transform this predominantly topological conjecture
here into log-concavity problems for the sum of convolutions
of some sequences called \textit{partial genus distributions}.
A simultaneous objective is to contribute new methods
to the theory of log-concavity.
Our topological graph-theoretic problem leads us to
consider special binary relations between
log-concave sequences, and also some special bivariate functions
that may appear  in collections of log-concave sequences.
We think that our newly developed concepts of
\textit{synchronicity}, \textit{ratio-dominance}, and \textit{lexicographicity}
are interesting in their own rights in log-concavity theory.
We develop several properties involving convolutions and
these notions, including some criteria for the log-concavity
of sums of sequences and for the log-concavity of sums of convolutions.

\subsection{{Historical background}}
Topological graph theory dates back to Heawood (1890) \cite{Hea1890}
and Heffter (1891) \cite{Hef1891}, who transformed the four-color
map problem for the plane, into the Heawood problem,
which specifies the maximum number of  map colors needed
for the surfaces of higher genus and for non-orientable surfaces as well.
The solution of this problem by Ringel and Youngs (1968) (see \cite{Ri74})
led to development of the study of embeddings of various families
of graphs in higher genus surfaces and to simultaneous study
of all the possible maps on a given surface.  There have subsequently
evolved substantial enumerative branches of graph embeddings,
according to genus, and maps, according to numbers of edges,
with frequent interplay between topological graph theory and
combinatorics, as in the present paper.

For instance, a combinatorial formula of Jackson~\cite{Jac87} (1987),
based on group character theory, is critical to the calculation  of
the genus  distributions of bouquets \cite{GrRoTu89} (bouquets
are graphs with one vertex and any number of self-loops).
Recent calculations of genus distributions of star-ladders \cite{ChGrMa12a}
and of embedding distributions of circular ladders \cite{ChGrMa12b}
use Chebyshev polynomials and the overlap matrix of Mohar \cite{Mo89}.
The log-concavity problem for genus distributions, our general target
here, is presently one of the oldest unsolved oft-cited problems in
topological graph theory.

There have always been two approaches to graph embeddings:
fix the graph and vary the surface, or fix the surface and vary the graph.
The Heawood Map Theorem can be viewed either way: find the largest
complete graph embeddable in the surface of \hbox{genus $g$},
or find the lowest genus surface in which a given complete graph
can be embedded.  Genus distribution, along with minimum and
maximum genus, are examples of fix-the-graph. Robertson and
Seymour's \cite{RS90} Kuratowski theorem for general surfaces,
which solves a problem of Erd\"os and K\"onig~\cite{Ko36}, is an
example of fix-the-surface.  Another example is the theory of maps,
where the graph embedding itself is fixed and its automorphisms are
investigated.  For example, until very recently, nothing was known
about the regular maps (maps having full rotational symmetry) in a
surface of given genus $g>5$;  a full classification of such maps
when $g=p+1$ for $p$ prime is given in \cite{CST}.  Methods there
are entirely algebraic, even involving parts of the classification of finite simple groups.

The study of graphs in surfaces has an uncanny history, beginning
with the Euler characteristic, of informing combinatorics, topology,
and algebra.  Kuratowski's Theorem leads  to the proof by
Robertson and Seymour~\cite{RS2} of the purely graph-theoretic
Wagner's Conjecture. Symmetries of maps leads, via Belyi's Theorem,
to Grothendieck's {\em dessins d'enfants} program~\cite{Gro} to study the
absolute Galois group of the rationals by its action on maps \cite{Gro,JoSi}.
It would not be surprising if genus distribution similarly informs enumerative combinatorics.

\subsection{Log-concave sequences}  %%%%%%%%%%%%%%%%%%%

Unimodal and log-concave sequences occur naturally in combinatorics,
algebra, analysis, geometry, computer science, probability and statistics.
We refer the reader to the survey papers of Stanley~\cite{Stan89} and Brenti~\cite{Br94}
for various results on unimodality and log-concavity.

The log-concavity of particular families of sequences and conditions that imply log-concavity have often been  studied before.  One such family is {\em P\'olya frequency sequences}.  By the Aissen-Schoenberg-Whitney theorem~\cite{ASW52}, the sequence $(a_k)_{k=0}^n$ is a P\'olya frequency sequence if and only if the polynomial $\sum_{k=0}^na_kx^k$ is real-rooted.  On the other hand, by a theorem of Newton (see, e.g., \cite{Br89}), the sequence of  coefficients of a real-rooted polynomial is log-concave. This provides an approach for proving the log-concavity of a sequence.

In fact, polynomials arising from combinatorics are often real-rooted; see~\cite{Kar68B,Pit97,Sch55,LW07,Stan00}.
It is also not uncommon to find classes of polynomials with some members real-rooted (and, thus, log-concave), yet with other members log-concave, despite imaginary roots.  For example, Wang and Zhao~\cite{WZ13} showed that all coordinator polynomials of Weyl group lattices are log-concave, while those of type $B_n$ are not real-rooted.  In this paper, we confirm the log-concavity of another well-studied sequence from topological graph theory, which was proved to be non-real-rooted, by using one of our criteria for log-concavity.

Recently some probabilists and statisticians care about the negative dependence of
random variables. According to Efron~\cite{Efr65} and Joag-Dev and Proschan~\cite{JoPr83},
if the independent random variables  $X_1$, $X_2$,  $\ldots$, $X_n$
have log-concave distributions, then their sum $X_1+X_2+\cdots+X_n$
is stochastically increasing, which in turn results in the negative association of the
distribution of $(X_1,X_2,\ldots,X_n)$ conditional on $\sum_{i\in A}X_i$,
for any nonempty subset $A$ of the set $\{1,2,\ldots,n\}$.

A recent paper of Huh \cite{Huh12} illustrates analogous interplay
between log-concavity and purely chromatic graph theory.
Huh proves the unimodality of the chromatic polynomial of any graph,
and thereby affirms a conjecture of Read \cite{Read68} and partially
affirms its generalization by Rota~\cite{Rota71}, Heron~\cite{Her72}
and Welsh~\cite{Wel76} into the context of matroids.

Other topics related to log-concavity include {\em $q$-log-concavity} (introduced by Stanley; see, e.g., \cite{But87,Kra89});  {\em strong $q$-log-concavity} (introduced by Sagan~\cite{Sag92});  {\em ultra log-concavity} (introduced by Pemantle~\cite{Pem00}; see also\cite{Lig97,WY07});  {\em $k$-log-concavity} and {\em $\infty$-log-concavity} (see~\cite{KP07,MS10});  {\em $q$-weighted log-concavity} (see~\cite{WZ08});  {\em ratio monotonicity} (see~\cite{CX09});  {\em reverse log-concavity} (see~\cite{CG09}); and so on.  Log-convex sequences have also received attention (e.g., \cite{CWY10,CTWY10}).
Some combinatorial proofs for them have emerged in turn (e.g., \cite{Sag88,CPQ12}).
Log-concavity of the convolution of sequences has been studied in~\cite[Section~6]{BBL09} implicitly.

In this paper, we introduce three new concepts regarding nonnegative
log-concave sequences.
A principal intent is to develop a tool to deal with a sum of convolutions
of log-concave sequences.
First, we introduce the binary relation \mdef{synchronicity}
of two log-concave sequences,
which is symmetric but not transitive;
it characterizes pairs of sequences that have synchronously non-increasing
ratios of successive elements.  As will be seen, synchronized sequences
form a monoid, under the usual addition operation.
Second, we introduce the binary relation \mdef{ratio-dominance}
between two synchronized log-concave sequences,
involving comparison of the ratios of successive terms of the same index.
We give some criteria for the ratio-dominance relation between two convolutions.
In particular, both synchronicity and ratio-dominance are
preserved by the convolution transformation associated with any
log-concave sequence.  Third, we examine collections of log-concave
sequences that admit a certain \mdef{lexicographic} condition,
which will be used to deal with the ratio-dominance relation between
two sums of convolutions of log-concave sequences.

\subsection{Genus distribution of a graph}\label{sec:GD}
  %%%%%%%%%%%%%%%%%%%

The \mdef{graph embeddings} we discuss are cellular and orientable.  Graphs are implicitly taken to be connected.  For general background in topological graph theory, see \cite{GrTu87, BW09B}.  Some prior acquaintance with partitioned genus distributions (e.g., \cite{GKP10, PoKhGr10}) would likely be helpful to a reader of this paper.

The \mdef{genus distribution} of a graph $G$ is the sequence $g_0(G)$, $g_1(G)$, $g_2(G)$, $\ldots$,
where $g_i(G)$ is the number of combinatorially distinct embeddings of~$G$
in the orientable surface of genus~$i$.  It follows from the \textit{interpolation theorem}
(see \cite{Duke66,BW09B}) that
any genus distribution contains only finitely many positive numbers and that
there are no zeros between the first and last positive numbers.

The earliest derivations \cite{FuGrSt89} of genus distributions were for closed-end ladders and for doubled-paths.  Genus distributions of bouquets, dipoles, and some related graphs were derived by \cite{GrRoTu89,KiLe98,Rie90}.  Genus distributions have been calculated more recently for various recursively specifiable sequences of graphs, including cubic outerplanar graphs~\cite{Gr11b}, 4-regular outerplanar graphs~\cite{PoKhGr11}, the $3$$\times$$n$-mesh $P_3\Box P_n$ \cite{KhPoGr12}, and cubic Halin graphs~\cite{Gr13}.  Some  calculations (e.g., \cite{ChMaZo11,ChMaZo12,ChGrMa12b}) also give the distribution of embeddings in non-orientable surfaces.

Proofs that the genus distributions of closed-end ladders and of doubled paths
are log-concave~\cite{FuGrSt89} were
based on closed formulas for those genus distributions.
Proof that the genus distributions of bouquets are log-concave~\cite{GrRoTu89}
was based on a recursion.

Stahl \cite{Stah97} used the term ``$H$-linear'' to describe chains of graphs obtained by amalgamating copies of a fixed graph $H$.  He conjectured that a number of these families of graphs have genus polynomials whose roots are real and nonpositive, which implies the log-concavity of their sequences of coefficients.  Although it was shown \cite{Wag97} that some of the families do indeed have such genus polynomials, Stahl's conjecture was disproved by Liu and Wang \cite{LW07}.

In particular, Example 6.7 of \cite{Stah97} is a sequence of $W_4$-linear graphs, in Stahl's terminology, where $W_4$ is the 4-wheel.  One of the genus polynomials of these graphs was proved to have non-real zeros in~\cite{LW07}.  We demonstrate in \S\ref{sec:iter-V}, nonetheless, that the genus distribution of every graph in this $W_4$-linear sequence is log-concave.  Thus, even though Stahl's proposed approach via roots of genus polynomials is insufficient, this paper does support Stahl's expectation that log-concavity of the genus distributions of chains of copies of a graph is a relatively accessible aspect of the general problem.  Genus distributions of several non-linear families of graphs are proved to be log-concave in \cite{GMT14}.

\subsection{Outline of this paper}  %%%%%%%%%%%%%%%%%%%%%%%

In \S\ref{sec:LC}, we define some possible relationships applying to sequences, and we then derive some purely combinatorial results regarding these relationships, which are used in \S\ref{sec:pgd} to establish the log-concavity of the genus distributions of graphs constructed by vertex- and edge-amalgamation operations.  We briefly review the theory of partitioned genus distributions at the outset of \S\ref{sec:pgd}.  We present recurrences in \S\ref{sec:iter-V} and \S\ref{sec:iter-E} for calculating the partitioned genus distributions of the graphs in chains of graphs joined iteratively by amalgamations.  We establish in these two subsections conditions on the partitioned genus distributions of the amalgamands under which the genus distribution of the vertex- and edge-amalgamated graphs, respectively, and their partial genus distributions are log-concave.

\medskip
%%%%%%%%%%%%%%%%%%%%%%%%%%%%%%%%%%%%%%%%%%%%%
%%%%%%%%%%%%%%%%%%%%%%%%%%%%%%%%%%%%%%%%%%%%%
\section{\large New Development of Log-Concave Sequences}\label{sec:LC}

We start by reviewing basic concepts and notation for log-concave sequences.  We say that a sequence $A=(a_k)_{k=0}^n$ is \mdef{nonnegative} if $a_k\ge0$ for all $k$.  An element $a_k$ is said to be an internal zero of~$A$ if there exist indices $i$ and $j$ with  $i<k<j$, such that $a_ia_j\ne 0$ and $a_k=0$. Throughout this paper, all sequences are assumed to be nonnegative and without internal zeros.

If $a_{k-1}a_{k+1} \le a_k^2$ for all $k$, then $A$ is said to be \mdef{log-concave}.  If there exists and index $h$ with $0\le h\le n$ such that
\[
a_0\,\le\, a_1\,\le\, \cdots\,\le\, a_{h-1}\,\le\, a_h\,\ge\, a_{h+1}\,\ge\,\cdots\,\ge\, a_n,
\]
then $A$ is said to be \mdef{unimodal}.  It is well-known that any nonnegative log-concave sequence without internal zeros is unimodal, and that any nonnegative unimodal sequence has no internal zeros.  Let \hbox{$B=(b_k)_{k=0}^m$} be another sequence.  The \mdef{convolution} of $A$ and $B$, denoted as $A*B$, is defined to be the coefficient sequence
\[
\Bigl(\sum\nolimits_{i=0}^ka_ib_{k-i}\Bigr)_{k=0}^{m+n}
\]
of the product of the polynomials
\[
f(x)=\sum\nolimits_{k=0}^na_kx^k
\rmand
g(x)=\sum\nolimits_{k=0}^mb_kx^k.
\]

To avoid confusion, we remark that in mathematical analysis, some people use the terminology ``convolution'' to mean the Hadamard product $\sum_{k}a_kb_kx^k$ of the functions~$f(x)$ and~$g(x)$, which is a topic quite different from ours.  By the Cauchy-Binet theorem, the convolution of two log-concave sequences without internal zeros is log-concave;  see~\cite{Men69} and~\cite[Proposition 2]{Stan89}.

For any finite sequence $A=(a_k)_{k=0}^n$,
we identify $A$ with the infinite sequence $(a_k')_{k=-\infty}^{\infty}$, where $a_k'=a_k$ for $0\le k\le n$, and $a_k'=0$ otherwise.  It is obvious that this identification is compatible with the definitions of unimodality, log-concavity, and convolution.  In the sequel, we will frequently employ inequalities of the form $\frac{a}{b}\le \frac{c}{d}$, where $a,b,c,d\ge0$. For convenience, we consider the inequality $\frac{a}{b}\le \frac{c}{d}$ to hold by default if $a=b=0$, or $c=d=0$, or $b=d=0$.

%\eject%%

\noi\textbf{Notation}.  We write $uA$ to denote the scalar multiple sequence $(ua_k)$, for any constant $u\ge0$.   The notation $A+B$ stands for the sequence $(a_k+b_k)$.  Greek letters $\alpha_k=a_k/a_{k-1}$ and $\beta_k=b_k/b_{k-1}$ denote ratios of successive terms of a sequence.   Thus, a sequence $A$ is log-concave if and only if the sequence $(\alpha_k)$ is non-increasing in $k$.  For $1\le i\le n$, let $A_i=(a_{i,k})_k$ be sequences.  Then we denote the indexed collection $(A_i)_{i=1}^n$ of sequences by $\A_n$.  For any sets $S$ and $T$, we write $S\le T$ (or $T\ge S$, equivalently) if $s\le t$ for all $s\in S$ and $t\in T$.

The following lemma will be useful in subsequent subsections.

\begin{lemma}\label{lem:sumpreserve} Suppose that for all $1\le i\le n$ and $1\le j\le m$,
we have $p_i$, $q_i$, $u_j$, $v_j\ge0$ and $\frac{p_i}{q_i} \,\le\, \frac{u_j}{v_j}$.  Then we have  \vsa\vsb
$$
\frac{\sum_{i=1}^np_i}{\sum_{i=1}^nq_i}  \,\le\, \frac{\sum_{j=1}^mu_j}{\sum_{j=1}^mv_j}.  $$
\end{lemma}

\begin{proof} \vsb
The desired inequality is equivalent to the inequality
\[
\sum_{i=1}^n\sum_{j=1}^m(p_iv_j-q_iu_j)\le 0,
\]
which is true because every summand is non-positive.
\end{proof}

In the next two subsections, we introduce the binary relations of synchronicity and ratio-dominance for log-concave sequences, and we give several properties regarding these relations and convolutions of log-concave sequences.  In \S\ref{sec:lex}, we introduce the concept of a lexicographic sequence and we establish a criterion (Theorem~\ref{thm:sumclvls}) for the ratio-dominance relation between sums of convolutions of log-concave sequences.

\smallskip
\subsection{The synchronicity relation} \label{sec:synch-def}   %%%%%%%%%%%%%

We say that  two nonnegative sequences $A$ and $B$ are \mdef{synchronized},
denoted as $A\sim B$, if both are log-concave, and they satisfy
$$ a_{k-1}b_{k+1}\,\le\,a_kb_k \rmand a_{k+1}b_{k-1}\,\le\,a_kb_k \fa{k}. $$
Alternatively, the synchronicity of $A$ and $B$ can be defined by the rule
\begin{equation}\label{def:syn}
   A\sim B \quad\Longleftrightarrow\quad
   \{\alpha_k,\beta_k\}\,\ge\,\{\alpha_{k+1},\beta_{k+1}\} \fa{k}.
\end{equation}
Therefore, $A\sim B$ implies $uA\sim vB$ for any $u,v\ge0$.  In other words, scalar multiplications preserve synchronicity.  It is clear that the synchronicity relation $\sim$ is symmetric.  We should be aware of that it is not transitive.

We denote by $\tLC$ the set of indexed collections of pairwise synchronized sequences.  Let $\A_n=(A_i)_{i=1}^n$. Then $\A_n\in\tLC$ if and only if $A_i\sim A_j$ for all $1\le i<j\le n$.

We will use the next lemma to prove the synchronicity of sums of synchronized sequences.

\eject

\begin{lemma}\label{lem:add}
Let the three sequences $A$, $B$, $C$  be log-concave and nonnegative.  If $(A,B,C)\in \tLC$, then $A+B\sim C$.
\end{lemma}

\begin{proof} \vsb
By Lemma~\ref{lem:sumpreserve} (or by direct calculation), we deduce that
\begin{align*}
\biggl\{{a_k+b_k\over a_{k-1}+b_{k-1}},\ {c_{k}\over c_{k-1}}\biggr\}
&\,\ge\,
\biggl\{{a_{k+1}+b_{k+1}\over a_k+b_k},\ {c_{k+1}\over c_k}\biggr\}
\fa{k}.
\end{align*}
By the synchronicity rule~(\ref{def:syn}), we infer that $A+B\sim C$.
\end{proof}

\begin{thm}\label{thm:lincomb}
Suppose that $\A_n\in \tLC$. For any numbers $u_1, v_1,u_2,v_2,\ldots,u_n,v_n\ge0$, we have
$\sum_{i=1}^n u_i A_i\sim\sum_{i=1}^n v_i A_i$.
\end{thm}

\begin{proof} \vsb
Since scalars preserve the synchronicity relation, we see that the $2n$ sequences $u_iA_i$ and~$v_iA_i$ are pairwise synchronized.  By iterative application of Lemma~\ref{lem:add}, we infer that the sequences $\sum_{i=1}^nu_iA_i$ and $\sum_{i=1}^nv_iA_i$ are synchronized.
\end{proof}

Now we show that convolution with the same log-concave sequence preserves synchronicity of two sequences.

\begin{thm}\label{thm:convo}
Let $A,B,C$ be three log-concave nonnegative sequences without internal zeros.  If $A\sim B$, then the convolution sequences $A*C$ and $B*C$ are synchronized.
\end{thm}

\begin{proof} \vsb
Since the convolution of two log-concave sequences without internal zeros is log-concave, it follows that the sequences $A*C$ and $B*C$ are log-concave.

To prove that the sequences $A*C$ and $B*C$ are synchronized, we arrange the terms of the product $(A*C)_{k-1}(B*C)_{k+1}$ into an array of $k$ rows of $k+2$ terms each, in which we number the rows and columns starting at zero, where row $j$ is
$$a_{j}b_0 c_{k-j-1}c_{k+1}\quad  a_{j}b_1 c_{k-j-1}c_{k}\quad \ldots\quad a_{j}b_{k+1} c_{k-j-1}c_{0}.$$
We similarly arrange the terms of the product $(A*C)_{k}(B*C)_{k}$ into an array of $k+1$ rows of $k+1$ terms each, in which row $j$ is
$$a_{j}b_0 c_{k-j}c_{k}\quad  a_{j}b_1 c_{k-j}c_{k-1}\quad \ldots\quad a_{j}b_{k} c_{k-j}c_{0}.$$
To prove that $(A*C)_{k-1}(B*C)_{k+1} \le (A*C)_{k}(B*C)_{k}$, we will demonstrate that the sum of the terms in the second array is at least as large as the sum of the terms in the first array.

We observe that the $k$ terms in column 0 of the first array are
$$a_0b_0 c_{k-1}c_{k+1}\quad  a_1b_0 c_{k-2}c_{k+1}\quad \ldots\quad a_{k-1}b_0 c_0c_{k+1}$$
and that the first $k$ terms in column 0 of the second array are
$$a_0b_0 c_k c_k \quad  a_1b_0 c_{k-1}c_k \quad \ldots\quad a_{k-1}b_0 c_1c_k.$$
Each term $a_jb_0c_{k-j-1}c_{k+1}$ is less than or equal to the corresponding term $a_jb_0c_{k-j}c_{k}$, since, by  log-concavity of $C$, respectively, we have
$$\frac{c_{k+1}}{c_k} \le \frac{c_{k}}{c_{k-1}} \le \cdots \le\frac{c_{k-j}}{c_{k-j-1}}.$$
We observe further that the $k$ terms in the $(k+1)$-st column of the first array are
$$a_0b_{k+1}c_{k-1}c_0\quad  a_1b_{k+1}c_{k-2}c_0\quad \ldots\quad a_{k-1}b_{k+1}c_0c_0$$
and that, excluding the term in column 0, the other $k$ terms in row $k$ of the second array are
$$a_kb_1c_0c_{k-1}\quad  a_kb_2c_0c_{k-2}\quad \ldots\quad a_kb_{k}c_0c_0.$$
The term $a_jb_{k+1}c_{k-j-1}c_0$ of the first array is less than or equal to the corresponding term $a_kb_{j+1}c_{k-j-1}c_0$ of the second array, since (by synchronicity of $A$ and $B$)
$$\frac{b_{k+1}}{a_k} \le \frac{b_{k}}{a_{k-1}}  \le \cdots \le \frac{b_{j+1}}{a_{j}}.$$
 by rows~$0$ through $k-1$ and columns 1 \hbox{through $k$}, with the sum of the entries in the square subarray of the second array formed by rows 0 through $k-1$ and columns 1 through $k$.

The terms on the main diagonals of these two square arrays are equal.  We now consider the sum  \begin{equation} \label{eq:sum1}
a_i b_j c_{k-1-i} c_{k+1-j} \,+\, a_{j-1}b_{i+1}c_{k-j}c_{k-i}
\end{equation}
of any term with $i>j$, and, hence, below the main diagonal of the first square subarray, and the term whose location is its reflection through that main diagonal.  We compare this to the sum
\begin{equation} \label{eq:sum2}
a_ib_jc_{k-i}c_{k-j} \,+\, a_{j-1}b_{i+1}c_{k+1-j}c_{k-1-i}
\end{equation}
of the two terms in the corresponding locations of the second square subarray.

By synchronicity of sequences $A$ and $B$, and since $j<i$, we have
$$\frac{b_{i+1}}{a_i} \le \frac{b_{i}}{a_{i-1}}  \le \cdots \le \frac{b_{j}}{a_{j-1}},$$
and thus, $a_{j-1}b_{i+1} \,\le\, a_ib_j$.
By log-concavity of $C$, we have
\[
c_{k-j}c_{k-i}-c_{k+1-j}c_{k-1-i} \ge 0.
\]
It follows that
\[
a_{j-1}b_{i+1}(c_{k-j}d_{k-i}-c_{k+1-j}d_{k-1-i})\le a_ib_j (c_{k-j}d_{k-i}-c_{k+1-j}d_{k-1-i}).
\]
We conclude that the sum \eqref{eq:sum1} is less than or equal to the sum \eqref{eq:sum2}, and, accordingly, that
$$(A*C)_{k-1}(B*C)_{k+1} \le (A*C)_{k}(B*C)_{k}.$$
A similar argument establishes that
$$(B*C)_{k-1}(A*C)_{k+1} \le (B*C)_{k}(A*C)_{k}.$$
Thus, the sequences $A*C$ and $B*C$ are synchronized.
\end{proof}

\smallskip
\subsection{The ratio-dominance relation}\label{sec:RD}  %%%%%%%%%%%

Let $A=(a_k)$ and $B=(b_k)$ be two nonnegative sequences.  We say that $B$ is \mdef{ratio-dominant} over~$A$, denoted as $B\gs A$ (or \hbox{$A\ls B$} equivalently), if $A\sim B$ and $a_{k+1}b_k\le a_kb_{k+1}$ for all $k$.  Alternatively, the ratio-dominance relation can be defined by
\begin{equation}\label{def:RD}
A\,\ls\, B  \quad\Longleftrightarrow\quad \beta_{k+1} \,\le\, \alpha_k \,\le\, \beta_k \fa{k}.
\end{equation}
It is clear that $A\ls B$ implies both $A$ and $B$ are log-concave, and $uA\ls vB$ for any $u,v\ge0$.  In other words, the scalars preserve the ratio-dominance relation.

From Definition \eqref{def:RD}, we can see that transitivity of the ratio-dominance relation holds if the sequences are pairwise synchronized.  Lemma \ref{lem:A:C}  will be used in proving Corollary~\ref{cor:AC:BD}.
\smallskip

\begin{lemma}\label{lem:A:C}
Let $A,B,C$ be three log-concave nonnegative sequences.  If $A\ls B$, $B\ls C$ and $A\sim C$, then $A\ls C$.   \qed
\end{lemma}

Before exploring basic properties of ratio-dominance, we mention two connections between our ratio-dominance relation and constructs that have been introduced in previous studies of log-concavity.
\smallskip

A nonnegative sequence $A=(a_k)$ is said to be {\em ultra-log-concave of order~$n$},
if the sequence $(a_k/\binom{n}{k})_{k=0}^n$ is log-concave, and $a_k=0$ for $k>n$.
So $A$ is ultra-log-concave of order~$\infty$ if the sequence~$(k!a_k)$ is log-concave;
see~\cite{Lig97}.  On the other hand,
by~(\ref{def:RD}),
the ratio-dominance of $(ka_k)$ over~$A$ reads as
$$ {(k+1)a_{k+1}\over ka_k}\le{a_k\over a_{k-1}}\le {ka_k\over (k-1)a_{k-1}}. $$
While the second inequality holds trivially, the first inequality represents
the log-concavity of the sequence $(k!a_k)$.
Therefore, the sequence $A$ is ultra-log-concave of order~$\infty$
if and only if the sequence $(ka_k)$ is ratio-dominant over $A$.
\smallskip

The other connection relates to Borcea et al.'s work~\cite[Section 6]{BBL09},
which studied the negative dependence properties of almost exchangeable measures.
More precisely, they considered log-concave sequences $A$ and $B$ such that
\begin{enumerate}[(i)]
\item\vsc the sequence $\bigl(\theta a_k+(1-\theta)b_k\bigr)$ is log-concave for all $0\le \theta\le 1$, and
\item $a_kb_{k+1}\ge a_{k+1}b_k$ for all $k$.
\end{enumerate}
It is clear that $A\sim B$ implies Condition~(i), and thus, $A\ls B$ implies both~(i) and~(ii).
As will be seen in the next section, to get the log-concavity of some genus distributions,
however, verifying the ratio-dominance relation
appears to be easier than checking~(i).
\smallskip

We now give some basic observations regarding the ratio-dominance relation.  First of all, we shall see by the following theorem that pairwise synchronized sequences can be characterized in terms of ratio-dominance relations.   Let $\A_n=(A_i)_{i=1}^n$ be a sequence of nonnegative sequences.

\noi\textbf{Notation}.  In what follows, we use $\alpha_{j,k}$ and $\beta_{j,k}$ to denote the ratios $a_{j,k}/a_{j,k-1}$ and  $b_{j,k}/b_{j,k-1}$, respectively.

\eject %%

\begin{thm} \label{thm:lc-equiv}
The following statements are equivalent.
\begin{itemize}
\item[(i)]  \vsb $\A_n\in \tLC$.
\item[(ii)]  There exists a sequence $B$ such that $\A_n\ls B$.
\item[(iii)]  There exists a sequence $C$ such that $\A_n\gs C$.
\end{itemize}
\end{thm}

\begin{proof} \vsb
We shall show equivalence of (i) and (ii).  The equivalence of (i) and (iii) is along the same line.  Suppose that $\A_n \in \tLC$.  Then let $k_0$ be the largest index~$k$ such that there exists an integer $i$ such that $1\le i\le n$, with
$$
a_{i,0}=a_{i,1}=\cdots=a_{i,k-1}=0\quad\text{ and }\quad a_{i,k}>0.
$$
Let $h_0$ be the largest index~$h$ such that there exists some $1\le i\le n$ with $a_{i,h}>0$.
Consider the sequence~$B$ defined by $b_{k_0}=1$ and
$$
\beta_k=\max_{1\le j\le n}\alpha_{j,k} \fa{k_0+1\le k\le h_0}.
$$
Let $1\le i,j\le n$.
The synchronicity $A_i\sim A_j$ implies that $\alpha_{j,k+1}\le\alpha_{i,k}$.  Thus,
\begin{equation}
\beta_{k+1}=\max_{1\le j\le n}\alpha_{j,k+1}\le\alpha_{i,k}.
\end{equation}
From Definition~(\ref{def:RD}), we deduce that $A_i\ls B$, which proves (ii). Conversely, suppose that $\A_n\ls B$.  Since \hbox{$A_j\ls B$}, we have $\alpha_{j,k+1}\le\beta_{k+1}$;  since \hbox{$A_i\ls B$}, we have $\beta_{k+1}\le\alpha_{i,k}$.  Thus $\alpha_{j,k+1}\le\alpha_{i,k}$. By symmetry, we have $\alpha_{i,k+1}\le\alpha_{j,k}$.  Therefore, $A_i\sim A_j$.
\end{proof}

\noi\textbf{Notation}.  The notations $\A_n\in \ltLC$ (and, respectively, $\A_n\in \gtLC$) if $A_i\ls A_j$ (resp., $A_i\gs A_j$), for all $1\le i<j\le n$, facilitate expression of a ratio-dominance analogue to \hbox{Theorem \ref{thm:lc-equiv}}.

\begin{thm}\label{thm:ls}
The following statements are equivalent.
\begin{itemize}
\item[(i)] \vsc $\A_n\in \ltLC$.
\item[(ii)] $A_1\ls A_2$, $A_2\ls A_3$, $\ldots$, $A_{n-1}\ls A_n$, and $A_1\sim A_n$.
\item[(iii)] $\alpha_{1,k}\le \alpha_{2,k}\le\cdots\le\alpha_{n,k}\le\alpha_{1,k-1}$ for every $k\ge1$.
\end{itemize}
\end{thm}

\begin{proof} \vsb
It is clear that (i) implies (ii).  Suppose that (ii) holds true.  Then, for any $1\le i\le n-1$, the relation $A_i\ls A_{i+1}$ implies $\alpha_{i,k}\le\alpha_{i+1,k}$; and the synchronicity $A_1\sim A_n$ implies $\alpha_{n,k}\le\alpha_{1,k-1}$.  This proves (iii). Now, suppose that (iii) holds true.  To show (i), it suffices to prove that $A_i\ls A_j$ for every pair $i,j$ such that $1\le i<j\le n$.  In fact,
$$
\alpha_{j,k+1}
\,\le\,
\alpha_{j+1,k+1}
\,\le\,
\cdots
\,\le\,
\alpha_{n,k+1}
\,\le\,
\alpha_{1,k}
\,\le\,
\alpha_{2,k}
\,\le\,
\cdots
\,\le\,
\alpha_{i,k}
\,\le\,
\cdots
\,\le\,
\alpha_{j,k}.
$$
In particular, we have $\alpha_{j,k+1}\le\alpha_{i,k}\le\alpha_{j,k}$.
By~(\ref{def:RD}), we deduce $\A_n\in \ltLC$. This completes the proof.
\end{proof}

Now we give one of the main results of in this section,
which will be used  in the next subsection to establish several ratio-dominance relations between convolutions.

\begin{thm}\label{thm:pwconvo}
Let $A,B,C,D$ be four nonnegative sequences without internal zeros.  If $A\lesssim B$ and $C\lesssim D$, then the convolution sequences $A*D$ and $B*C$ are synchronized.
\end{thm}

\begin{proof}\vsb
Let $A*D=(s_k)$ and $B*C=(t_k)$.  We also define
$$ L=\frac{s_{n+1}}{ s_{n}} \rmand R=\frac{t_{n}}{t_{n-1}}. $$
To prove Theorem \ref{thm:pwconvo}, it is sufficient to establish the inequality $L\le R$, since the inequality $s_{n-1}t_{n+1}\le s_nt_n$ will then follow by symmetry.

Toward that objective, we first observe that
\begin{eqnarray*}
\frac{\partial L}{\partial d_0} &=& \frac{a_{n+1}s_n-a_ns_{n+1}}{ s_n^2} \\
&=& \frac{1}{s_n^2}\biggl(\sum_{i=0}^nd_i(a_{n-i}a_{n+1}-a_{n+1-i}a_n)-a_0a_nd_{n+1}\biggr) \;\le\; 0.
\end{eqnarray*}
Therefore,
\begin{equation}\label{ineq:conj26:d0}
L\,\le\, \frac{s_{n+1}}{s_{n}} \bigg|_{d_0=0}   = \quad
\frac{ { \sum_{i=0}^{n}a_id_{n+1-i} }_{\vphantom{|}} }
{ { \sum_{i=0}^{n-1} {a_i d_{n-i}} }^{\vphantom{W}} }.
\end{equation}
For any $0\le k\le n$, we now define
$$ u_k\,=\,\frac{1}{b_{n-k}}\sum_{j=1}^k b_{n-j}d_j
\rmand
v_k\,=\,\frac{1}{b_{n-k}}\sum_{j=0}^k b_{n-j}d_{j+1} $$
and we define \vskip-18pt
\begin{equation}  \label{eq:g-def}
g_k \;=\; \frac{{\sum_{i=k+1}^{n}a_{n-i}d_{i+1} \;+\; v_ka_{n-k}}_{\vphantom{|}}  }
{{ \sum_{i=k+1}^na_{n-i}d_i \;+\; u_ka_{n-k}}^{\vphantom{|}} }.\hskip1cm
\end{equation}
Then Inequality~\eqref{ineq:conj26:d0} reads $L\le g_0$, and we pause here in the proof of Theorem \ref{thm:pwconvo} for the following lemma.

\smallskip
\begin{lemma}\label{lem:ineq2:b}
For any number $k$ such that $0\le k\le n-1$, we have $g_k\le g_{k+1}$, where $g_k$ is defined by Equation \eqref{eq:g-def}.
\end{lemma}

\begin{proof}\vsb
Let $0\le k\le n-1$. Denote the sum in the denominator of $g_k$ by $M$. Then
\begin{align*}
\frac{\partial g_k}{\partial a_{n-k}}
&\;=\; \frac{1}{M^2}
   \biggl(v_k\sum_{i=k+1}^na_{n-i}d_i-u_k\sum_{i=k+1}^{n}a_{n-i}d_{i+1}\biggr)\\
&\;=\;\frac{1}{M^2}
   \sum_{i=k+1}^{n}a_{n-i}(v_kd_i-u_kd_{i+1})\\
&\;\ge\;\frac{1}{M^2}
   \sum_{i=k+1}^{n} \sum_{j=1}^k \,
\frac{ {a_{n-i}b_{n-j}}_{\vphantom{w}} } { b_{n-k} } (d_id_{j+1}-d_{i+1}d_j),
\end{align*}
which is nonnegative, by the log-concavity of $D$.
Since $A\lesssim B$, we deduce that
\begin{eqnarray*}
g_k &\le&g_k \bigg|_{a_{n-k} = \frac {b_{n-k}a_{n-k-1}} {b_{n-k-1}}}\\
&=& \frac{\sum_{i=k+1}^{n}a_{n-i}d_{i+1} \;+\; v_k \frac{b_{n-k}}{b_{n-k-1}} a_{n-k-1}
}{
\sum_{i=k+1}^na_{n-i}d_i \;+\; u_k \frac{b_{n-k}}{ b_{n-k-1}} a_{n-k-1}}
~=~ g_{k+1}.
\end{eqnarray*}
This completes the proof of Lemma \ref{lem:ineq2:b}.
\end{proof}

From Inequality \eqref{ineq:conj26:d0} and Lemma~\ref{lem:ineq2:b}, we deduce that
\[
L\le g_0\le g_1\le\cdots\le g_n=\frac{v_n}{u_n}
~=~ \frac{\sum_{j=0}^{n}b_{j}d_{n-j+1} }{\sum_{j=1}^{n}b_{j-1}d_{n-j+1}}.
\]
Then, to prove that $L\le R$, it suffices to show
\begin{equation}\label{ineq:conj26:desired}
\sum_{i=0}^n b_ic_{n-i} \sum_{j=0}^{n}b_{j-1}d_{n-j+1}
~\ge~
\sum_{i=0}^n b_{i-1}c_{n-i} \sum_{j=0}^{n}b_jd_{n-j+1}.
\end{equation}
The difference between the sides of Inequality \eqref{ineq:conj26:desired} is
\begin{align*}
& \hskip-6pt \sum_{0\le i\le n}\;\sum_{0\le j\le n} c_{n-i}d_{n-j+1}(b_ib_{j-1}-b_{i-1}b_j)\\
=& \sum_{0\le i<j\le n} \hskip-6pt c_{n-i}d_{n-j+1}(b_ib_{j-1}-b_{i-1}b_j)
   \;+  \hskip-6pt \sum_{0\le j<i\le n} \hskip-6pt c_{n-i}d_{n-j+1}(b_ib_{j-1}-b_{i-1}b_j)\\
=& \sum_{0\le i<j\le n} \hskip-6pt c_{n-i}d_{n-j+1}(b_ib_{j-1}-b_{i-1}b_j)
   \;+  \hskip-6pt \sum_{0\le i<j\le n} \hskip-6pt c_{n-j}d_{n-i+1}(b_jb_{i-1}-b_{j-1}b_i)\\
=& \sum_{0\le i<j\le n}(b_ib_{j-1}-b_{i-1}b_j)(c_{n-i}d_{n-j+1}-c_{n-j}d_{n-i+1}).
\end{align*}
Since $B$ is log-concave and $0\le i<j\le n$, we have $b_ib_{j-1}-b_{i-1}b_j\ge0$.
On the other hand, since $C\lesssim D$, we derive that
\[
\frac{c_{n-i}}{c_{n-j}}
~= \prod_{k=n-j+1}^{n-i} \frac{c_{k}}{c_{k-1}}
~\ge~ \prod_{k=n-j+1}^{n-i} \frac{d_{k+1}}{d_{k}}
~=~ \frac{d_{n-i+1}}{ d_{n-j+1}}.
\]
This proves Inequality \eqref{ineq:conj26:desired}, and it thereby completes the proof of Theorem \ref{thm:pwconvo}.
\end{proof}

We caution that $A\lesssim B$ and $C\lesssim D$ do not together imply that $A+C\lesssim B+D$.
Next is an immediate corollary of Theorem~\ref{thm:pwconvo}.

\begin{cor}\label{cor:ls:syn}
Let $\A_n$ and $\B_n$ be two sequences of nonnegative sequences without internal zeros.  Suppose that $\A_n\in\ltLC$ and $\B_n\in\gtLC$, then the sequences $A_k*B_k$ are pairwise synchronized. In other words, we have $A_i*B_i\sim A_j*B_j$ for any $1\le i\le j\le n$.
\end{cor}

For any sequences $\A_n$ and $\B_m$ of nonnegative sequences,
we write $\A_n\ls\B_m$ if $A_i\ls B_j$ for any $1\le i\le n$ and $1\le j\le m$.

\begin{thm}\label{thm:sumls}
Let $\A_n$ and $\B_m$ be two sequences of nonnegative sequences.  If \hbox{$\A_n\ls\B_m$}, then we have
$$ \sum_{i=1}^n u_iA_i  \,\ls\,  \sum_{j=1}^m v_jB_j $$
for any $u_1,u_2,\ldots,u_n,v_1,v_2,\ldots,v_m\ge0$.
\end{thm}

\begin{proof} \vsb
Let $u_1,u_2,\ldots,u_n,v_1,v_2,\ldots,v_m\ge0$.  By Lemma~\ref{lem:sumpreserve}, we have
\begin{equation}
{\sum_{j=1}^mv_jb_{j,k+1}\over\sum_{j=1}^mv_jb_{j,k}}
\,\le\,
{\sum_{i=1}^nu_ia_{i,k}\over\sum_{i=1}^nu_ia_{i,k-1}}
\,\le\,
{\sum_{j=1}^mv_jb_{j,k}\over\sum_{j=1}^mv_jb_{j,k-1}}
\fa{k}.
\end{equation}
This proves the desired ratio-dominance relation, by Definition~(\ref{def:RD}).
\end{proof}

Theorem \ref{thm:sumls} will be used to prove the ratio-dominance relation between sums of sequences. The inverse statement does not hold, unless there are some synchronicity relations between the sequences.  We give the simplest case as follows, which can be proved directly from the definitions.

\begin{thm}\label{thm:simls:ls}
If $A\sim B$ and $A\ls A+B$, then $A\ls B$.
\end{thm}

We mention that the ratio-dominance relation $A\ls B$ or $A\gs B$ implies the synchronicity $A\sim B$, which, in turn, implies the log-concavity of $A$ and of $B$.  This leads to the following two sufficiency criteria for the log-concavities of sequences, which are obtained respectively from Theorem~\ref{thm:lincomb} and Corollary~\ref{cor:ls:syn}.

\begin{thm}
Suppose that $\A_n\in\tLC$. For any $u_1,u_2,\ldots,u_n\ge0$,
the sequence $\sum_{i=1}^n u_iA_i$ is log-concave.
\end{thm}

\begin{thm}
Suppose that $\A_n\in\ltLC$ and $\B_n\in\gtLC$. For any $u_1,u_2,\ldots,u_n\ge0$,
the sequence $\sum_{i=1}^n u_iA_i*B_i$ is log-concave.
\end{thm}

We remark that Theorem~\ref{thm:lincomb},
Theorem~\ref{thm:pwconvo}, Corollary~\ref{cor:ls:syn} and Theorem~\ref{thm:sumls} are not only interesting in their own rights, but also potentially useful toward the conjectures in \S\ref{sec:concl} regarding log-concavity of partial genus distributions of graphs.

\medskip
\subsection{Lexicographic bivariate functions} \label{sec:lex}  %%%%%%%%%%

For every $1\le i\le n$, let $F_i=(f_{i,k})_k$ be a nonnegative sequence.  We distinguish the sequence $F_i$ from the sequence $A_i$ by allowing $f_{i,k}=+\infty$.  More precisely, we suppose that $0\le f_{i,k}\le+\infty$ for any~$i$ and~$k$.  Regarding $f_{i,k}$ as bivariate functions, we say that the finite sequence $(F_i)_{i=1}^n$ of sequences is \mdef{lexicographic} if
$$  f_{i,k}\le f_{j,h}\quad\Longleftrightarrow\quad
\text{either $k=h$ and $i\le j$, or $k<h$}.
$$
Equivalently, the lexicographicity of the bivariate function $f_{i,k}$ can be defined by
\begin{equation}\label{def:lx}
f_{1,k}\,\le\, f_{2,k}\,\le\, \cdots\,\le\, f_{n,k}\,\le\, f_{1,k+1}  \fa{k}.
\end{equation}
That is, the entries in each column of a representation of \DWc{$(f_{i,k})$} as an array are
in non-decreasing order, and they are less than or equal to every entry in subsequent columns.

\begin{thm}\label{thm:sumclvls}
Let $\A_n,\B_n,\W_n$ be three sequences of nonnegative sequences without internal zeros.  Suppose that \hbox{$\W_n\in \ltLC$}, and the bivariate functions $b_{i,t}/a_{i,t}$ and $a_{i,t-1}/b_{i,t}$ are both lexicographic.  Then we have
\begin{equation}\label{eq0}
\sum_{i=1}^nW_i*A_i ~\ls~ \sum_{i=1}^nW_i*B_i.
\end{equation}
In other words, Inequality (\ref{eq0}) holds true for any $\W_n\in \ltLC$ if
\begin{equation}\label{lex:b/a}
\frac{b_{1,t}}{a_{1,t}} \,\le\, \frac{b_{2,t}}{a_{2,t}} \,\le\, \cdots \, \le\, \frac{b_{n,t}}{a_{n,t}}\,\le\, \frac{b_{1,t+1}}{a_{1,t+1}}
\end{equation} \vsa
and \vsa\vsb
\begin{equation}\label{lex:a/b}
 \frac{a_{1,t-1}}{ b_{1,t}} \,\le\, \frac{a_{2,t-1}}{ b_{2,t}} \,\le\, \cdots \,\le\,
\frac{a_{n,t-1}}{ b_{n,t}} \,\le\, \frac{a_{1,t}}{ b_{1,t+1}}
\end{equation}
for all $t$.
\end{thm}

\begin{proof} \vsb
By Definition~(\ref{def:RD}), Relation~(\ref{eq0}) is equivalent to
\begin{equation}\label{ineq0}
\frac{\sum_{i=1}^n\sum_{s}w_{i,s}b_{i,k+1-s}}
	{\sum_{j=1}^n\sum_{t}w_{j,t}b_{j,k-t}} \,\le\,
\frac{\sum_{i=1}^n\sum_{s}w_{i,s}a_{i,k-s}}
	{\sum_{j=1}^n\sum_{t}w_{j,t}a_{j,k-1-t}} \,\le\,
\frac{\sum_{i=1}^n\sum_{s}w_{i,s}b_{i,k-s}}
	{\sum_{j=1}^n\sum_{t}w_{j,t}b_{j,k-1-t}}.
\end{equation}
We shall now deal with these two  respective inequalities.

The first inequality in~(\ref{ineq0}) is equivalent to the inequality
\begin{equation}\label{eq1}
\sum_{1 \,\le\,  i,j \,\le\,  n}\sum_{s,t}
w_{i,s}w_{j,t-1}(a_{j,k-t}b_{i,k+1-s}-a_{i,k-s}b_{j,k+1-t}) ~\le~ 0.
\end{equation}
When $i=j$, the inner summation in~(\ref{eq1}) becomes
\begin{align*}
\sum_{s,t} & w_{i,s}w_{i,t-1} (a_{i,k-t}b_{i,k+1-s}-a_{i,k-s}b_{i,k+1-t})\\
& =~ \sum_{s<t} (w_{i,s}w_{i,t-1}-w_{i,t}w_{i,s-1})(a_{i,k-t}b_{i,k+1-s}-a_{i,k-s}b_{i,k+1-t}).
\end{align*}
Let $s<t$. Since $W_i$ is log-concave, we have
$$ w_{i,s}w_{i,t-1}-w_{i,t}w_{i,s-1}
\,\ge\,0. $$
On the other hand, we have ${a_{i,k}\over a_{i,k-1}}\le {b_{i,k}\over b_{i,k-1}}$ by~(\ref{lex:b/a}), and we have ${b_{i,k+1}\over b_{i,k}}\le {a_{i,k}\over a_{i,k-1}}$. So $A_i\ls B_i$. We infer that
$$ a_{i,k-t}b_{i,k+1-s}-a_{i,k-s}b_{i,k+1-t} \,\le\, 0. $$
Therefore, the $i=j$ part of the summand in~(\ref{eq1}) is nonpositive. For the $i\ne j$ part, we can make the transformation
{\allowdisplaybreaks
\begin{align*}
&\sum_{1\le i\ne j\le n}\!\sum_{s,t}
	w_{i,s}w_{j,t-1}(a_{j,k-t}b_{i,k+1-s}-a_{i,k-s}b_{j,k+1-t}) \\
=& \sum_{1\le i\ne j\le n}\!\sum_{s,t}
	w_{i,k+1-s}w_{j,k-t}(a_{j,t-1}b_{i,s}-a_{i,s-1}b_{j,t}) \\
=&\sum_{1\le i<j\le n}\!\sum_{s,t}
	\Bigl[w_{i,k+1-s}w_{j,k-t}(a_{j,t-1}b_{i,s}-a_{i,s-1}b_{j,t})
	+w_{j,k+1-s}w_{i,k-t}(a_{i,t-1}b_{j,s}-a_{j,s-1}b_{i,t})\Bigr]\\
=&\sum_{1\le i<j\le n}\!\sum_{s,t}
	\Bigl[w_{i,k+1-s}w_{j,k-t}(a_{j,t-1}b_{i,s}-a_{i,s-1}b_{j,t})
	+w_{j,k+1-t}w_{i,k-s}(a_{i,s-1}b_{j,t}-a_{j,t-1}b_{i,s})\Bigr]\\
=&\sum_{1\le i<j\le n}\!\sum_{s,t}
	(w_{i,k+1-s}w_{j,k-t}-w_{i,k-s}w_{j,k+1-t})(a_{j,t-1}b_{i,s}-a_{i,s-1}b_{j,t}).
\end{align*}
}
Let $i<j$. Then $W_i\ls W_j$, which implies that
$$w_{i,k+1-s}w_{j,k-t}-w_{j,k+1-t}w_{i,k-s}$$
is nonpositive (resp., nonnegative) if $s\le t$ (resp., $s>t$). Therefore, (\ref{eq1}) holds true if $a_{j,t-1}b_{i,s}-a_{i,s-1}b_{j,t}$ is nonnegative (resp., nonpositive) for all $i<j$, when $s\le t$ (resp., $s>t$).   Since $A_j\sim B_j$, the function $b_{j,t}/a_{j,t-1}$ in $t$ is non-increasing.  Therefore, Condition \eqref{ineq0} is equivalent to the iterated inequality
\begin{equation}\label{eq51}
{b_{i,t+1}\over a_{i,t}}\le{b_{j,t}\over a_{j,t-1}}\le{b_{i,t}\over a_{i,t-1}}
\fa{1\le i<j\le n}.
\end{equation}

Along the same line, the second inequality in~(\ref{ineq0}) holds if
\begin{equation}\label{eq61}
{b_{i,t}\over a_{i,t}} \,\le\, {b_{j,t}\over a_{j,t}} \,\le\, {b_{i,t+1}\over a_{i,t+1}}
\fa{1\le i<j\le n}.
\end{equation}
Combining~(\ref{eq51}) and~(\ref{eq61}), we find that they can be recast as
\begin{equation}\label{eq7}
{b_{i,t}\over b_{j,t}}
\,\le\,
{a_{i,t}\over a_{j,t}}
\,\le\,
{b_{i,t+1}\over b_{j,t+1}}
\fa{1\le i<j\le n}
\end{equation}
and
\begin{equation}\label{eq8}
{b_{i,t+1}\over b_{j,t}}
\,\le\,
{a_{i,t}\over a_{j,t-1}}
\,\le\,
{b_{i,t}\over b_{j,t-1}}
\fa{1\le i<j\le n}.
\end{equation}

Now we are going to further reduce the above inequality system. Note that~(\ref{eq7}) holds for all $i<j$ if and only if it holds for all $j=i+1$, that is,
\begin{equation}\label{eq71}
{b_{i,t}\over b_{i+1,t}}
\,\le\,
{a_{i,t}\over a_{i+1,t}}
\,\le\,
{b_{i,t+1}\over b_{i+1,t+1}}
\fa{1\le i\le n-1},
\end{equation}
because one may get~(\ref{eq7}) by multiplying~(\ref{eq71}) by
the inequalities obtained from itself by substituting~$i$ to $i+1,i+2,\ldots,j-1$.
On the other hand,
the first inequality in~(\ref{eq8}) can be rewritten as
\begin{equation}\label{eq8-1}
\frac{a_{j,t-1}}{ b_{j,t}} \,\le\, \frac{a_{i,t}}{ b_{i,t+1}}.
\end{equation}
By the second inequality in~(\ref{eq71}), the function $a_{i,t}/b_{i,t+1}$
in $i$ is non-decreasing.
So~(\ref{eq8-1}) holds for all $i<j$ if and only if it holds for $i=1$ and $j=n$,
that is,
\begin{equation}\label{eq8-11}
\frac{a_{n,t-1}}{ b_{n,t}} \,\le\, \frac{a_{1,t}}{ b_{1,t+1}}.
\end{equation}
Similarly, the second inequality in~(\ref{eq8})
is equivalent to
\begin{equation}\label{eq8-21}
\frac{b_{n,t-1}}{ a_{n,t-1}} \,\le\, \frac{b_{1,t}}{ a_{1,t}}.
\end{equation}
The inequalities~(\ref{eq8-11}) and~(\ref{eq8-21}) can be written together as
\begin{equation}\label{eq81}
\frac{b_{1,t+1}}{ b_{n,t}} \,\le\,  \frac{a_{1,t}}{ a_{n,t-1}} \,\le\, \frac{b_{1,t}}{ b_{n,t-1}}
\fa{t}.
\end{equation}

Finally, it is clear that the combination of (\ref{eq71}) and (\ref{eq81}) is equivalent to the two lexicographical conditions.
\end{proof}

Using $n=2$, Theorem~\ref{thm:sumclvls} will be applied in the next section toward proving the log-concavity of the genus distributions of some families of graphs. When $n=1$, Theorem~\ref{thm:sumclvls} reduces to the following extension of Theorem~\ref{thm:convo}, i.e., the convolution transformation induced from the same log-concave sequence preserves ratio-dominance.

\begin{cor}\label{cor:AC:BC}
Let $A,B,C$ be log-concave nonnegative sequences without internal zeros.
If $A\ls B$, then $A*C\ls B*C$.
\end{cor}

Theorem~\ref{thm:sumclvls} also implies the next result on the ratio-dominance relation between convolutions.

\begin{cor}\label{cor:AC:BD}
Let $A,B,C,D$ be nonnegative sequences without internal zeros.
If $(A,B,C,D)\in\ltLC$, then $A*C\ls B*D$.
\end{cor}

\begin{proof} \vsb
Since $A\ls D$ and $B\ls C$, we deduce $A*C\sim B*D$ by Theorem~\ref{thm:pwconvo}.
By Corollary~\ref{cor:AC:BC}, we have $A*C\ls B*C$ and $B*C\ls B*D$.
Now the desired relation follows immediately from Lemma~\ref{lem:A:C}.
\end{proof}

To provide a better interpretation of the lexicographic conditions in the statement of Theorem~\ref{thm:sumclvls}, the next theorem gives some of its consequences.

\begin{thm}\label{thm:LX}
Let $\A_n$ and $\B_n$ be two sequences of nonnegative sequences.  Suppose that both the functions $b_{i,t}/a_{i,t}$ and $a_{i,t-1}/b_{i,t}$ are lexicographic.  Then $\A_n\in\gtLC$, $\B_n\in\gtLC$ and $A_i\ls B_i$ for all $1\le i\le n$.
\end{thm}

\begin{proof} \vsb
By the definition~(\ref{def:RD}), the relation $A_{i}\gs A_{i+1}$ is equivalent to
\begin{equation}\label{ineq:remark}
{a_{i,t+1}\over a_{i,t}}
\,\le\,
{a_{i+1,t}\over a_{i+1,t-1}}
\,\le\,
{a_{i,t}\over a_{i,t-1}}
\fa{t}.
\end{equation}
Note that all the transformations beyond Relation~(\ref{eq61}) are equivalences. The first inequality in~(\ref{ineq:remark}) can be seen from~(\ref{eq8}), while the second is clear from~(\ref{eq71}). This proves that $\A_n\in \gtLC$.  For the same reason, we have $\B_n\in \gtLC$. As aforementioned,
the relation $A_i\ls B_i$ can be seen from~(\ref{lex:b/a}) and~(\ref{lex:a/b}). This completes the proof.
\end{proof}

\subsection{The offset sequence}

In this subsection, we introduce a concept to be used in the next section,
which is interesting in its own right.
For any sequence $A=(a_k)$, we define the associated \mdef{offset sequence}
$A^+=(a_k^+)$ by $a_k^+=a_{k-1}$ for all~$k$.
It is easy to prove the following proposition.

\begin{prop}\label{prop:+}
Let $A$ and $B$ be nonnegative sequences. Then we have the following equivalence
relations:
\[
A\,\ls\, B\ \Longleftrightarrow\
B\,\ls\, A^+\ \Longleftrightarrow\
A^+\,\ls\, B^+.
\]
In particular, $A\ls A^+$.    \qed
\end{prop}

We remark that $A\ls B$ does not imply $A\sim B^+$.  The next theorem
is a corollary of Theorem~\ref{thm:pwconvo}, with a proof similar to the last part of the proof of Theorem~\ref{thm:pwconvo}.

\begin{thm}\label{thm:BC:AC+}
Let $A,B,C$ be nonnegative log-concave sequences without internal zeros.
If $A\ls B$, then $B*C\ls A*C^+$.
\end{thm}

\begin{proof} \vsb
By Proposition~\ref{prop:+}, we have $C\ls C^+$.
Applying Theorem~\ref{thm:pwconvo}, we deduce $B*C\sim A*C^+$.
Now it suffices to show that \[
{\sum_j b_jc_{k-j}\over\sum_j b_jc_{k-1-j}}
\le
{\sum_i a_ic_{k-i}^+\over\sum_i a_ic_{k-1-i}^+}.
\]
Equivalently,
\[
\sum_i a_ic_{k-i-1}\sum_j b_jc_{k-1-j}
\ge
\sum_i a_ic_{k-i-2}\sum_j b_jc_{k-j}.
\]
The difference between the sides of the above inequality
can be transformed as
\begin{align*}
&\sum_{i}\sum_{j}
a_ib_j(c_{k-i-1} c_{k-1-j}-c_{k-i-2}c_{k-j})\\
=&\sum_{i}\sum_{j}
a_{i-1}b_j(c_{k-i} c_{k-1-j}-c_{k-i-1}c_{k-j})\\
=&\sum_{i<j}
(a_{i-1}b_j-a_{j-1}b_i)(c_{k-i} c_{k-1-j}-c_{k-i-1}c_{k-j}).
\end{align*}
Let $i<j$. Since $A\ls B$, we have
$a_{i-1}b_j-a_{j-1}b_i\le 0$.
On the other hand, the log-concavity of
the sequence $C$ implies $c_{n-i} c_{n-1-j}-c_{n-i-1}c_{n-j}\le 0$.
Therefore, every summand in the above sum is nonnegative.
This completes the proof.
\end{proof}

\medskip
%%%%%%%%%%%%%%%%%%%%%%%%%%%%%%%%%%%%%%%%%%%%%
%%%%%%%%%%%%%%%%%%%%%%%%%%%%%%%%%%%%%%%%%%%%%
\section{\large Partial Genus Distributions} \label{sec:pgd}

We denote the oriented surface of genus $i$ by $S_i$.
For any vertex-rooted (or edge-rooted) graph $(G,y)$, with $y$ either a $2$-valent vertex or an edge with two $2$-valent endpoints,  we define

\begin{tabular}{ccl}
$g_i(G)$ &:& the number of embedding $G\to S_i$. \\[2pt]
$d_i(G,y)$ &:& the number of embeddings $G\to S_i$, such that \\[2pt]
&& two different face-boundary walks are incident on $y$.  \\ [2pt]
$s_i(G,y)$ &:& the number of embeddings $G\to S_i$ such that a single \\[-1pt]
&& \quad face-boundary walk is twice incident on root $y$.  \\[3pt]
\end{tabular} \vsb

\noi We observe that $g_i(G) \,=\, d_i(G,y)+s_i(G,y)$.
\smallskip

We define the \mdef{genus distribution polynomial} as the power series \vskip-15pt
$$
\Gamma(G)(x)~=~g_0(G) + g_1(G)x + g_2(G)x^2 + \cdots
$$ \vsb
and the \mdef{partial genus distribution polynomials} by
\begin{align*}
D(G,y)(x) &=  d_0(G,y) + d_1(G,y)x + d_2(G,y)x^2 + \cdots, \\
S(G,y)(x) &=  s_0(G,Y) + s_1(G,Y)x + s_2(G,y)x^2 + \cdots.
\end{align*} \vsb
\noi We say that the genus distribution $\Gamma(G)$ is partitioned into the partial genus distributions $D(G,y)$ and $S(G,y)$, and we refer to the pair $\{D(G,y),  S(G,y)\}$ as a \mdef{partitioned genus distribution}.

More generally we allow any subgraph of a graph to serve as a root.  The number of partial genus distributions needed for calculations involving amalgamations rises rapidly with the number of vertices and edges in the root subgraph.  The best understood instances are when the root comprises two vertices or two edges.  We say then that the graph is \mdef{doubly rooted}.  Just as a single-root partial genus distribution refines a total genus distribution, a double-root partial genus distribution refines a single-root distribution.

We observe that genus distributions are integer sequences with non-negative terms, only finitely many non-zero terms, and no internal zeros (due to the ``interpolation theorem'').
 It is not presently known whether partial genus distributions can have internal zeros.

\noi\textbf{Remark.}
Although calculating the minimum genus of a graph is known to be NP-hard, it has been proved that for any fixed treewidth and bounded degree, there is a quadratic-time algorithm \cite{Gr12b} for calculating genus distributions and partial genus distributions.  Nonetheless, it is not a practical algorithm.  The values in the partial genus distributions given in this paper were calculated by a ``brute force'' computer program created by Imran Khan.  Using symmetries, it is not difficult to calculate the partial genus distribution values used in Examples \ref{eg:W4-chain}, \ref{eg:ML4-chain}, and \ref{eg:K4-chain} by hand.  A hand calculation for Example \ref{eg:circ7:1,2-chain} would be tedious.

\subsection{Iterative amalgamation at vertex-roots} \label{sec:iter-V}

In this subsection, we present a theorem that will enable us to establish conditions sufficient for all the members of a sequence of graphs constructed by iterative vertex-amalgamation to have log-concave genus distributions.  For amalgamations involving a doubly vertex-rooted graph $(G,u,v)$ with two
$2$-valent roots, the following partitioning of the embedding counting variable $g_i(G)$  is described in \cite{Gr11a}:
$$\begin{matrix}
dd^0_i(G,u,v) &  dd'_i(G,u,v) & dd''_i(G,u,v) & ds^0_i(G,u,v) & ds'_i(G,u,v) \\[3pt]
sd^0_i(G,u,v) & sd'_i(G,u,v) & ss^0_i(G,u,v) & ss^1_i(G,u,v) & ss^2_i(G,u,v).
\end{matrix}$$
We use two upper-case letters to denote the corresponding double-root partial genus distributions, for which purpose we signify the graph by a subscript.  For instance, $DS'_{(G,u,v)}$ denotes the sequence
$$ds'_0(G,u,v),\ ds'_1(G,u,v),\ ds'_2(G,u,v),\ \ldots$$
It is sometimes convenient to use the groupings
\begin{eqnarray*}
dd_i(G,u,v) &=& dd^0_i(G,u,v) + dd'_i(G,u,v) + dd''_i(G,u,v), \\
ds_i(G,u,v) &=& ds^0_i(G,u,v) + ds'_i(G,u,v), \\
sd_i(G,u,v) &=& sd^0_i(G,u,v) + sd'_i(G,u,v), \\
ss_i(G,u,v) &=&  ss^0_i(G,u,v) + ss^1_i(G,u,v) +  ss^2_i(G,u,v).
\end{eqnarray*}
When we form a linear chain of copies of a graph $G$, we suppress the first root in the initial  copy of $G$, and we define
\begin{eqnarray*}
D_{(G,v)} &=& DD_{(G,u,v)} + SD_{(G,u,v)}, \\
S_{(G,v)} &=& DS_{(G,u,v)} + SS_{(G,u,v)}.
\end{eqnarray*}

\begin{thm}  \label{thm:chain-V}
{\allowdisplaybreaks
Let $(G, t)$ be a vertex-rooted graph and $(H, u, v)$ a doubly vertex-rooted graph, where all roots are $2$-valent.  Let $(X,v)$ be the vertex-rooted graph obtained from the disjoint union $G\sqcup H$ by merging vertex $t$ with vertex $u$.  Then the following recursion holds true:
\begin{eqnarray}
 D_{(X,v)} &=& \hskip10pt 4D_{(G,t)}*DD_{H(u,v)} \,+\, 2D_{(G,t)}*DD^{0+}_{H(u,v)} \label{def:vertex:D}\\
&&+\, 2D_{(G,t)}*DD'^{+}_{H(u,v)}  \,+\, 6S_{(G,t)}*DD_{H(u,v)} \notag \\
&&+\, 6D_{(G,t)}*SD_{H(u,v)}   \,+\, 6S_{(G,t)}*SD_{H(u,v)}  \notag\\
&&+\, 2D_{(G,t)}*SS^2_{H(u,v)},   \notag\\
%\noalign{\eject}
S_{(X,v)} &=& \hskip10pt  2D_{(G,t)}*DD''^{+}_{H(u,v)} \,+\, 4D_{(G,t)}*DS_{H(u,v)} \label{def:vertex:S}\\
&&+\, 2D_{(G,t)}*DS'^{+}_{H(u,v)} \,+\, 6S_{(G,t)}*DS_{H(u,v)}  \notag\\
&&+\, 6D_{(G,t)}*SS^0_{H(u,v)} \,+\, 6D_{(G,t)}*SS^1_{H(u,v)} \notag \\
&&+\, 4D_{(G,t)}*SS^{2}_{H(u,v)}  \,+\, 6S_{(G,t)}*SS_{H(u,v)}. \notag
 \end{eqnarray}
}
\end{thm}
\begin{proof}\vsb
This theorem is a form of Corollary 3.8 of \cite{GKP10}.
\end{proof}

\begin{thm} \label{thm:rec-V}
Let $(G,t)$ be a vertex-rooted graph such that $D_{(G,t)} \lesssim S_{(G,t)}$.  Let $(H,u,v)$ be a doubly vertex-rooted graph.  We introduce the following abbreviations:
\begin{equation}\label{def:AB:V}
\begin{aligned}
A_1\hsb\ &= DD^{0+}_{(H,u,v)}+\hsc DD'^+_{(H,u,v)}+\hsc SS^2_{(H,u,v)}+\hsc SD_{(H,u,v)},\\
A_2\hsb\ &= SD_{(H,u,v)}+\hsc DD_{(H,u,v)},\\
B_1\hsb\ &= DS'^+_{(H,u,v)}+\hsc DD''^+_{(H,u,v)}+\hsc SS^0_{(H,u,v)}+\hsc SS^1_{(H,u,v)}, \\
B_2\hsb\ &= SS_{(H,u,v)}+\hsc DS_{(H,u,v)}.
\end{aligned}
\end{equation}
Suppose that
\begin{equation} \label{cond:vertex:ls}
{b_{1,t}\over a_{1,t}}\,\le\,{b_{2,t}\over a_{2,t}}\,\le\,{b_{1,t+1}\over a_{1,t+1}}
\quad\text{ and }\quad
{a_{1,t-1}\over b_{1,t}}\,\le\,{a_{2,t-1}\over b_{2,t}}\,\le\,{a_{1,t}\over b_{1,t+1}}
\fa{t}.
\end{equation}
Then the partial genus distributions $D_{(X,v)}$ and $S_{(X,v)}$
of the vertex-rooted graph $(X,v)$, obtained from $G\sqcup H$ by merging vertices $t$ and $u$, are both log-concave;  moreover,  we have $D_{(X,v)}\lesssim S_{(X,v)}$.
\end{thm}

\begin{proof} \vsb
In view of~(\ref{def:vertex:D}) and~(\ref{def:vertex:S}),
we observe that
\begin{equation}\label{fm:X-DS}
\begin{split}
D_{(X,v)}&=~ 2D_{(G,t)}*A_1+(4D_{(G,t)}+6S_{(G,t)})*A_2,\\
S_{(X,v)}&=~ 2D_{(G,t)}*B_1+(4D_{(G,t)}+6S_{(G,t)})*B_2.
\end{split}
\end{equation}
By Theorem~\ref{thm:sumls} and the premise $D_{(G,t)} \lesssim S_{(G,t)}$, we have the relation
\[
2D_{(G,t)} \lesssim 4D_{(G,t)} + 6S_{(G,t)}.
\]
Therefore, we deduce $D_{(X,v)}\lesssim S_{(X,v)}$ by Theorem \ref{thm:sumclvls} directly.
This completes the proof.
\end{proof}
\smallskip

\begin{cor} \label{cor:V-chain}
Let $(G,t)$ be a vertex-rooted graph such that $D_{(G,t)} \lesssim S_{(G,t)}$.
Let $\bigl( (H_k,u_k,v_k)\bigr)_{k=1}^n$ be a sequence of doubly vertex-rooted graphs,
whose partial genus distributions satisfy Relation~\eqref{cond:vertex:ls}.
Then the iteratively vertex-amalgamated graph
$$(X_n,v_n) \,=\, (G,t)*(H_1,u_1,v_1)*\cdots*(H_n,u_n,v_n)$$
has log-concave partial genus distributions  $D_{(X_n,v_n)}$ and $S_{(X_n,v_n)}$, and $D_{(X_n,v_n)}\lesssim S_{(X_n,v_n)}$.  Moreover, the genus distribution $\Gamma(X)$ is log-concave.
\end{cor}

\begin{proof} \vsb
The proof is by iterative application of Theorem \ref{thm:rec-V} and application of Lemma~\ref{lem:add}  to the sum $\Gamma(X)= D_{(X_n,v_n)}+S_{(X_n,v_n)}$.
\end{proof}
\smallskip

\begin{example} \label{eg:W4-chain}
Let $(G,t)$ be the 4-wheel with a root vertex inserted at the midpoint of a rim edge.  Let $(H,u,v)$ be the 4-wheel with a vertex inserted at the midpoint of each of two non-adjacent rim edges, as illustrated in Figure \ref{fig:W4*W4}.    This example was previously discussed by Stahl \cite{Stah97}.  \smallskip\vskip4pt

\begin{figure} [ht]
\centering  \vsb
    \includegraphics[width=5in]{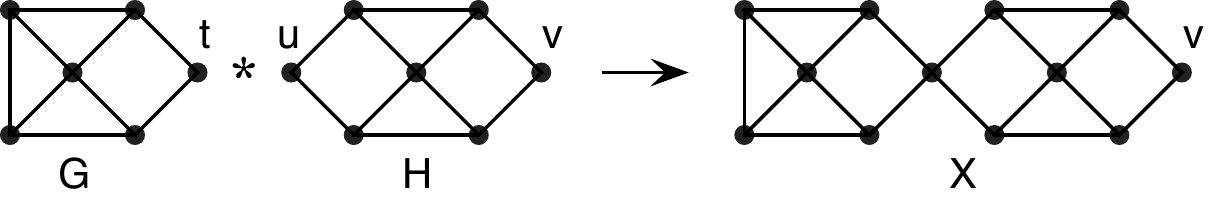}  \vsb
\caption{Starting a $W_4$-chain with vertex-amalgamation.}
\label{fig:W4*W4}
\end{figure}  %\vskip-12pt

The vertex-rooted graph $(G,t)$ has the partitioned genus distribution \vsa\vsc
$$ D_{(G,t)} \,=\, (2,\, 44) \rmand  S_{(G,t)} \,=\, (0,\, 14,\, 36).  $$

The doubly vertex-rooted graph $(H,u,v)$ has the partitioned genus distribution
\begin{alignat*}{3}
DD^0_{(H,u,v)}&=(0),\qquad&
DD'_{(H,u,v)}&=(2,20), \qquad&
DD''_{(H,u,v)}&=(0,12),\\
DS^0_{(H,u,v)}&=(0,4),\qquad&
DS'_{(H,u,v)}&=(0,8),&&\\
SD^0_{(H,u,v)}&=(0,4),\qquad&
SD'_{(H,u,v)}&=(0,8),&&\\
SS^0_{(H,u,v)}&=(0),\qquad&
SS^1_{(H,u,v)}&=(0,0,24),\qquad&
SS^2_{(H,u,v)}&=(0,2,12).
\end{alignat*}
We may group for convenience.
\begin{alignat*}{2}
DD_{(H,u,v)}&=(2,32),\qquad&
DS_{(H,u,v)}&=(0,12),\\
SD_{(H,u,v)}&=(0,12),\qquad&
SS_{(H,u,v)}&=(0,2,36).
\end{alignat*}
By the definition~(\ref{def:AB:V}),
it is straightforward to calculate
\[
A_1=(0,16,22),\qquad
A_2=(2,44),\qquad
B_1=(0,0,44),\qquad
B_2=(0,14,36).
\]
None of them has internal zeros.
Since any inequality of the form ${0\over0}\le{x\over y}$ is considered to be true,
the lexicographical conditions in~\eqref{cond:vertex:ls} reduce to
\[
\frac{b_{2,0}}{a_{2,0}} \,\le\,
\frac{b_{1,1}}{a_{1,1}} \,\le\,
\frac{b_{2,1}}{a_{2,1}} \,\le\,
\frac{b_{1,2}}{a_{1,2}} \,\le\,
\frac{b_{2,2}}{a_{2,2}}\quad\text{ and }\quad
\frac{a_{2,0}}{ b_{2,1}} \,\le\,
\frac{a_{1,1}}{ b_{1,2}}\,\le\,
\frac{a_{2,1}}{b_{2,2}}\,\le\,
\frac{a_{1,2}}{b_{1,3}},
\]
which we verify as
$$
{0\over2}\,\le\,
{0\over16}\,\le\,
\frac{14}{44} \,\le\,
\frac{44}{32} \,\le\,
\frac{36}{0}
\quad\text{ and }\quad
\frac{2}{14}\,\le\,
\frac{16}{44} \,\le\,
\frac{44}{36}\,\le\,
{22\over0}.
$$
Thus, we anticipate from Theorem  \ref{thm:rec-V} that  $D_{(X,v)} \ls S_{(X,v)}$ and, accordingly, that $\Gamma_{(X,v)}$ is log-concave.  Using Theorem \ref{thm:chain-V}, we calculate
\begin{align*}
D_{(X,v)} &= (16,\, 936,\, 13408,\, 12320),\\
S_{(X,v)} &= (0,\, 144,\, 4776,\, 15552, 7776),\\
\Gamma(X) &= (16,\, 1080,\, 18184,\, 27872,\, 7776).
\end{align*}
The condition that $D_{(X,v)}\ls S_{(X,v)}$ is verified as follows:
\[
{16\over0}
\ge{144\over0}
\ge{936\over 16}
\ge{4776\over144}
\ge{13408\over 936}
\ge{15552\over 4776}
\ge{12320\over 13408}
\ge{7776\over 15552}
\ge{0\over12320}.
\]
It is easy to verify that $\Gamma(X)$ is indeed log-concave.

Note that any inequality of the form ${0\over0}\le{x\over y}$ is considered to be true.
Therefore, as a consequence of Corollary~\ref{cor:V-chain}, despite the Liu-Wang disproof of  Stahl's conjecture that the roots of the genus polynomials of all these $W_4$-chains are real, we conclude that every $W_4$-chain constructed by iterative vertex-amalgamation has a log-concave genus distribution.  Moreover, the single-root partials of every $W_4$-chain are log-concave and synchronized.
\end{example}
\smallskip

\begin{example} \label{eg:ML4-chain}
This time, let $(G,t)$ be the M\"obius ladder $ML_4$ with a root-vertex created at the midpoint of any edge.   Let $(H,u,v)$ be $ML_4$ with two root-vertices created at the midpoints of antipodal edges of $K_4$, as illustrated in Figure \ref{fig:ML4-chain}.

\begin{figure} [ht]
\centering  \vsb
    \includegraphics[width=2.5in]{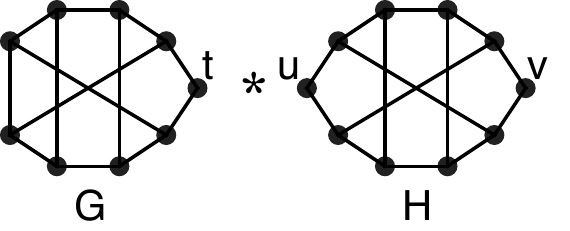}  \vsb
\caption{Ingredients for an $ML_4$-chain with iterative vertex-amalgamation.}
\label{fig:ML4-chain}
\end{figure}
The vertex-rooted graph $(G,t)$ has the partitioned genus distribution
\[
D_{(G,t)} \,=\, (0,\, 48,\, 96)
\rmand
S_{(G,t)} \,=\, (0,\, 8,\, 104).
\]
The doubly vertex-rooted graph $(H,u,v)$ has the partitioned genus distribution
\begin{alignat*}{3}
DD^0_{(H,u,v)}&=(0,8),\qquad&
DD'_{(H,u,v)}&=(0,20), \qquad&
DD''_{(H,u,v)}&=(0,14,\, 48),\\
DS^0_{(H,u,v)}&=(0,2),\qquad&
DS'_{(H,u,v)}&=(0,4,48),&&\\
SD^0_{(H,u,v)}&=(0,2),\qquad&
SD'_{(H,u,v)}&=(0,4,48),&&\\
SS^0_{(H,u,v)}&=(0,0,8),\qquad&
SS^1_{(H,u,v)}&=(0,0,28),\qquad&
SS^2_{(H,u,v)}&=(0,2,20).
\end{alignat*}
We may group for convenience.
\begin{alignat*}{2}
DD_{(H,u,v)}&=(0,42,48),\qquad&
DS_{(H,u,v)}&=(0,6,48),\\
SD_{(H,u,v)}&=(0,6,48),\qquad&
SS_{(H,u,v)}&=(0,2,56).
\end{alignat*}
%\[
%{2230\over0}\ge{12288\over0}\ge{43200\over 2304}\ge{82176\over 12288}
%\ge{128256\over43200}
%\ge{115200\over82176}
%\ge{9216\over128256}
%\ge{0\over115200}.
%\]
By Definition~(\ref{def:AB:V}),
we calculate
$$
A_1=(0,\,8,\,96),\quad
A_2=(0,\,48,\,96),\quad
B_1=(0,\,0,\,54,\,96),\quad
B_2=(0,\,8,\,104).
$$
None of them has internal zeros. Moreover, in this case, the lexicographical conditions in~(\ref{cond:vertex:ls}) reduce to
\[
\frac{b_{1,1}}{a_{1,1}} \,\le\,
\frac{b_{2,1}}{a_{2,1}} \,\le\,
\frac{b_{1,2}}{a_{1,2}} \,\le\,
\frac{b_{2,2}}{a_{2,2}} \,\le\,
\frac{b_{1,3}}{a_{1,3}}\quad\text{ and }\quad
\frac{a_{2,0}}{ b_{2,1}} \,\le\,
\frac{a_{1,1}}{ b_{1,2}}\,\le\,
\frac{a_{2,1}}{b_{2,2}} \,\le\,
\frac{a_{1,2}}{ b_{1,3}} \,\le\,
\frac{a_{2,2}}{b_{2,3}}.
\]
which we verify as
$$ \frac{0}{8} \,\le\,
\frac{8}{48} \,\le\,
\frac{54}{96} \,\le\,
\frac{104}{96} \,\le\,
\frac{96}{0}  \quad\text{ and }\quad
\frac{0}{8} \,\le\,
\frac{8}{54} \,\le\,
\frac{48}{104} \,\le\,
{96\over96} \,\le\,
{104\over0}.
$$
Using Theorem \ref{thm:chain-V}, we calculate
\begin{align*}
D_{(X,v)} &= (0,\, 0,\, 12288,\, 82176,\, 115200),\\
S_{(X,v)} &= (0,\, 0,\, 2304,\, 43200,\, 128256,\, 9216),\\
\Gamma(X) &=\, (0,\, 0,\, 14592,\, 125376,\, 243456,\, 9216).
\end{align*}
It is again easy to verify that $D_{(X,v)}\ls S_{(X,v)}$, and that $\Gamma(X)$ is log-concave.  We conclude that every $ML_4$-chain constructed by iterative vertex-amalgamation has log-concave genus distribution and log-concave single-root partial genus distributions.
\end{example}
\smallskip

\begin{example}
By combining Example \ref{eg:W4-chain} and Example \ref{eg:ML4-chain}, we see that if $X$ is a vertex-amalgamation chain of copies of $W_4$ and $ML_4$, interspersed arbitrarily with each other, then the genus distribution $\Gamma(X)$ is log-concave.  Moreover, the single-root partials are log-concave and synchronized.
\end{example}

\smallskip
\subsection{Iterative amalgamation at edge-roots} \label{sec:iter-E}  %%%%%%%%

This section presents the edge amalgamation analogy to the vertex-amalgamation discussion in \S\ref{sec:iter-V}.  When a graph has two edge-roots and both endpoints of both edge-roots are $2$-valent, the partitioning is similar to the case of two $2$-valent vertex-roots.  However, the recursions used for constructing linear chains of copies of a graph have different coefficients.  Definitions of the double-edge-rooted partials are given in \cite{PoKhGr10}. A key difference from vertex-amalgamation is that the two ways of merging two root edges can lead to non-isomorphic graphs with the same partial genus distributions.

\begin{thm}  \label{thm:chain-E}
{\allowdisplaybreaks
Let $(G, e)$ be a single-edge-rooted graph and $(H, g, f)$ a double-edge-rooted graph, where each edge-root has two $2$-valent endpoints.  Let $(W,f)$ be the single edge-rooted graph obtained from the disjoint union $G\sqcup H$ by merging edge~$e$ with edge~$g$.  Then the following recursions hold true:
\begin{eqnarray}
 D_{(W,f)} &=& \hskip3pt 2D_{(G,e)}*DD_{H(g,f)} \,+\, 2D_{(G,e)}*DD^{0+}_{H(g,f)} \\
&&+\,  2D_{(G,e)}*DD'^+_{H(g,f)}  \,+\, 4D_{(G,e)}*SD_{H(g,f)}   \notag\\
&&+\,  2D_{(G,e)}*SS^2_{H(g,f)}  \,+\, 4S_{(G,e)}*DD_{H(g,f)}  \notag\\
&&+\, 4S_{(G,e)}*SD'_{H(g,f)},  \notag\\
S_{(W,f)} &=& \hskip3pt  2D_{(G,e)}*DD''^{+}_{H(g,f)} \,+\, 2D_{(G,e)}*DS_{H(g,f)} \\
&&+\, 2D_{(G,e)}*DS^{+}_{H(g,f)} \,+\, 4D_{(G,e)}*SS^{0}_{H(g,f)}  \notag\\
&&+\, 4D_{(G,e)}*SS^{1}_{H(g,f)}  \,+\, 2D_{(G,e)}*SS^2_{H(g,f)}  \notag\\
&&+\, 4S_{(G,e)}*DS_{H(g,f)}  \,+\, 4S_{(G,e)}*SS_{H(g,f)}. \notag
 \end{eqnarray}
}
\end{thm}

\begin{proof}  \vsc
This theorem is a corollary of Theorems 3.2, 3.3, and 3.4 (collectively) of \cite{PoKhGr10}.
\end{proof}

It is easy to derive the next result, analogous to Theorem~\ref{thm:rec-V}, if one notices that

\begin{align*}
\noalign{\vsa\vsb}
D_{(W,f)}&=~ 2D_{(G,e)}*A_1+(2D_{(G,e)}+4S_{(G,e)})*A_2,\\
S_{(W,f)}&=~ 2D_{(G,e)}*B_1+(2D_{(G,e)}+4S_{(G,e)})*B_2.
\end{align*}
%\smallskip

\begin{thm} \label{thm:rec-E}
Let $(G,e)$ be an edge-rooted graph such that $D_{(G,e)} \lesssim S_{(G,e)}$.
Let $(H,g,f)$ be a doubly edge-rooted graph. We introduce the following abbreviations:
\begin{equation}\label{def:AB:E}
\begin{aligned}
A_1\hsb\ &= DD^{0+}_{(H,g,f)}+\hsc DD'^+_{(H,g,f)}+\hsc\ SS^2_{(H,g,f)} +\hsc SD^0_{(H,g,f)} +\hsc SD_{(H,g,f)},\\
A_2\hsb\ &= SD'_{(H,g,f)} +\hsc DD_{(H,g,f)},\\
B_1\hsb\ &= DS^+_{(H,g,f)} +\hsc DD''^+_{(H,g,f)} +\hsc SS^0_{(H,g,f)} +\hsc SS^1_{(H,g,f)}, \\
B_2\hsb\ &= SS_{(H,g,f)} +\hsc DS_{(H,g,f)}.
\end{aligned}
\end{equation}
Suppose that
\begin{equation} \label{cond:edge:ls}
{b_{1,t}\over a_{1,t}}\le{b_{2,t}\over a_{2,t}}\le{b_{1,t+1}\over a_{1,t+1}}
\quad\text{ and }\quad
{a_{1,t-1}\over b_{1,t}}\le{a_{2,t-1}\over b_{2,t}}\le{a_{1,t}\over b_{1,t+1}} \fa{t}.
\end{equation}
Then the partial genus distributions  $D_{(W,f)}$ and $S_{(W,f)}$ of the edge-rooted graph $(W,f)$, obtained from $G\sqcup H$ by merging edges $e$ and $g$, are both log-concave;  moreover,
we have $D_{(W,f)}\lesssim S_{(W,f)}$.  \qed
\end{thm}

We observe that Relation~(\ref{cond:vertex:ls}) and Relation~(\ref{cond:edge:ls})
are of the same form.
\smallskip

\begin{cor} \label{cor:E-chain}
Let $(G,e)$ be an edge-rooted graph such that the partial distributions $D_{(G,t)}$ and $S_{(G,t)}$ are log-concave and that $D_{(G,e)} \lesssim S_{(G,e)}$.
Let $\bigl((H_k,u_k,v_k)\bigr)_{k=1}^n$ be a sequence of doubly edge-rooted graphs whose partial genus distributions are all log-concave and satisfy Relation \eqref{cond:edge:ls}.  Then the iteratively edge-amalgamated graph
$$(W_n,f_n) \,=\, (G,e)*(H_1,g_1,f_1)*\cdots*(H_n,g_n,f_n)$$
has log-concave partial genus distributions  $D_{(W_n,f_n)}$ and $S_{(W_n,f_n)}$, and $D_{(W_n,f_n)}\lesssim S_{(W_n,f_n)}$.  Moreover, the genus distribution $\Gamma(W)$ is log-concave.
\end{cor}

\begin{proof} \vsb
The proof is by iterative application of Theorem \ref{thm:rec-V} and application of Lemma~\ref{lem:add}  to the sum $\Gamma(W)= D_{(W_n,f_n)}+S_{(W_n,f_n)}$.
\end{proof}
\smallskip

\begin{example}  \label{eg:K4-chain}
Let $(G,e)$ be the complete graph $K_4$ with a root-edge created as the middle segment of a trisection of any edge of $K_4$.   Let $(H,g,f)$ be $K_4$ with two root-edges created as the middle segments of non-adjacent edges of $K_4$, as illustrated in Figure \ref{fig:K4-chain}.

\begin{figure} [ht]
\centering  \vsb
    \includegraphics[width=5in]{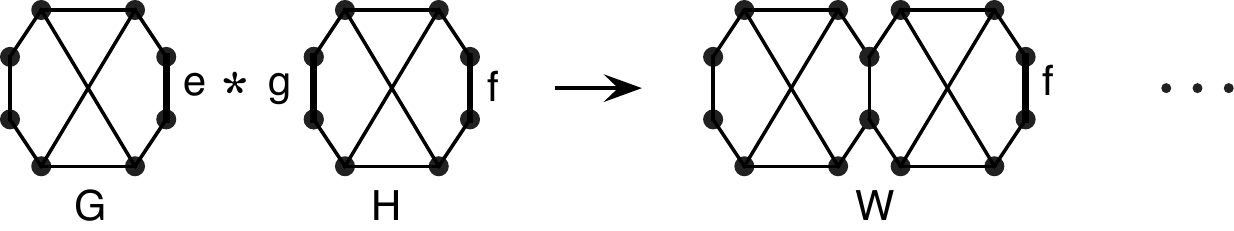}  \vsb
\caption{Forming a $K_4$-chain by iterative edge-amalgamation.}
\label{fig:K4-chain}
\end{figure}
\end{example} \vskip-12pt

The edge-rooted graph $(G,e)$ has the partitioned genus distribution
$$ D_{(G,e)}= (2,8)
\rmand
S_{(G,e)}= (0,6). $$
The doubly edge-rooted graph $(H,g,f)$ has the following non-zero partial genus distributions:
\[
DD^0_{(H,g,f)}= (2),\quad
DD''_{(H,g,f)}=SD'_{(H,g,f)}=DS'_{(H,g,f)}= (0,4),\quad
SS^2_{(H,g,f)}= (0,2).
\]
Using Theorem \ref{thm:chain-E}, we can calculate
\begin{align*}
D_{(W,f)}&=(8,\,144,\,448),\\
S_{(W,f)}&=(0,\,24,\,272,\,128),\\
\Gamma(W)&=(8,\,168,\,720,\,128).
\end{align*}
We easily verify that $\Gamma(W)$ is log-concave and that $D_{(W,f)}\ls S_{(W,f)}$.
By~(\ref{def:AB:E}), we can easily compute
$$
A_1=(0,8),\quad
A_2=(2,8),\quad
B_1=(0,0,8),\quad
B_2=(0,6).
$$
The lexicographical conditions in~(\ref{cond:edge:ls}) reduce to
\[
\frac{b_{2,0}}{a_{2,0}}\,\le\,
\frac{b_{1,1}}{a_{1,1}}\,\le\,
\frac{b_{2,1}}{a_{2,1}}\,\le\,
\frac{b_{1,2}}{a_{1,2}}\quad\text{ and }\quad
\frac{a_{2,0}}{b_{2,1}}\,\le\,
\frac{a_{1,1}}{b_{1,2}}\,\le\,
\frac{a_{2,1}}{b_{2,2}}.
\]
In this example, they are
\[
{0\over2}\,\le\,
{0\over8}\,\le\,
{6\over8}\,\le\,
{8\over0}\quad\text{ and }\quad
{2\over6}\,\le\,
{8\over8}\,\le\,
{8\over0}.
\]
We conclude that every $K_4$-chain constructed by iterating edge-amalgamations has a log-concave genus distribution.  Moreover, the single-root partials are log-concave and synchronized.

\begin{example}  \label{eg:circ7:1,2-chain}
Let $(G,e)$ be the circulant graph $circ(7\,$:$\,1,2)$ with root-edges created as middle segments of  trisections of edges, as shown in Figure \ref{fig:circ7:1,2-chain}. Then the edge-rooted graph $(G,e)$ has the partitioned genus distribution
$$ D_{(G,e)}= (0,\,492,\,25642,\,120960) \rmand S_{(G,e)}= (0,\,0,\,2694,\,61352,\,68796). $$

\begin{figure} [ht]
\centering  \vsb
    \includegraphics[width=3in]{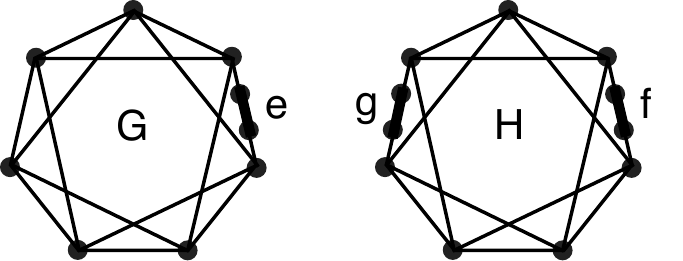}  \vsb
\caption{The circulant graph $circ(7\,$:$\,1,2)$ with root-edges created as shown.}
\label{fig:circ7:1,2-chain}
\end{figure}
\end{example}

The doubly edge-rooted graph $(H,g,f)$ has the following non-zero partial genus distributions:
\begin{xalignat*}{2}
DD^0_{(H,g,f)}&=(0,\,382,\,9296),&
DD'_{(H,g,f)}&=(0,\,110,\,12564,\,53476),\\
DD''_{(H,g,f)}&=(0,\,0,\,1162,\,24400),&
DS^0_{(H,g,f)}&=(0,\,0,\,1476,\,8740),\\
DS'_{(H,g,f)}&=(0,\,0,\,1144,\,34344),&
SD^0_{(H,g,f)}&=(0,\,0,\,1476,\,8740),\\
SD'_{(H,g,f)}&=(0,\,0,\,1144,\,34344),&
SS^0_{(H,g,f)}&=(0,\,0,\,0,\,3584),\\
SS^1_{(H,g,f)}&=(0,\,0,\,0,\,7268,\,39328),&
SS^2_{(H,g,f)}&=(0,\,0,\,74,\,7416,\,29468).
\end{xalignat*}
Using definition~(\ref{def:AB:E}), we can compute
\begin{xalignat*}{2}
A_1 &=(0,\,0,\,4662,\,81100,\,82944),&
A_2 &= (0,\,492,\,24166,\,112220), \\
B_1 &= (0,\,0,\,0,\,14634,\,106812),&
B_2 &= (0,\,0,\,2694,\,61352,\,68796).
\end{xalignat*}
The lexicographical conditions in~(\ref{cond:edge:ls}) reduce to
\[
\begin{split}
&\frac{b_{2,1}}{a_{2,1}}\,\le\,
\frac{b_{1,2}}{a_{1,2}}\,\le\,
\frac{b_{2,2}}{a_{2,2}}\,\le\,
\frac{b_{1,3}}{a_{1,3}}\,\le\,
\frac{b_{2,3}}{a_{2,3}}\,\le\,
\frac{b_{1,4}}{a_{1,4}}\,\le\,
\frac{b_{2,4}}{a_{2,4}},\\[3pt]
&\frac{a_{2,1}}{b_{2,2}}\,\le\,
\frac{a_{1,2}}{b_{1,3}}\,\le\,
\frac{a_{2,2}}{b_{2,3}}\,\le\,
\frac{a_{1,3}}{b_{1,4}}\,\le\,
\frac{a_{2,3}}{b_{2,4}}\,\le\,
\frac{a_{1,4}}{b_{1,5}}.
\end{split}
\]
and they can be verified as
\begin{alignat*}{6}
\frac{0}{492}&\le
\frac{0}{4662}&\le
\frac{2694}{24166}&\le
\frac{14634}{81100}&\le
\frac{61352}{112220}&\le
\frac{106812}{82944}&\le
\frac{68796}{0},\\[4pt]
\frac{492}{2694}&\le
\frac{4662}{14634}&\le
\frac{24166}{61352}&\le
\frac{81100}{106812}&\le
\frac{112220}{68796}&\le
\frac{82944}{0}.
\end{alignat*}
We conclude that every one of these $circ(7\,$:$\,1,2)$-chains has a log-concave genus distribution. Moreover, the single-root partials are log-concave and synchronized.  This example illustrates that the properties we need to apply these new methods are not restricted to very small graphs.

\medskip
%%%%%%%%%%%%%%%%%%%%%%%%%%%%%%%%%%%%%%%%%%%%%
%%%%%%%%%%%%%%%%%%%%%%%%%%%%%%%%%%%%%%%%%%%%%
\section{\large Conclusions} \label{sec:concl}

We have introduced some new methods for proving the log-concavity of a linear combination of log-concave sequences and of log-concave sequences that have been transformed by convolutions.  We have used these methods to show that linear chains of graphs that satisfy certain conditions known to be true of many graphs, and which are possibly true for all graphs, have log-concave genus distributions.  This motivates further study and development of these new methods for application to proving the log-concavity of the genus distributions of larger classes of graphs.

We have proved that, given a collection of graphs that are doubly vertex-rooted or doubly edge-rooted, whose partitioned genus distributions satisfy conditions given in Corollary \ref{cor:V-chain} or Corollary \ref{cor:E-chain}, respectively, a linear chain formed from those graphs by iterative amalgamation has a log-concave genus distribution and log-concave partial genus distributions, as well.

We offer two restricted forms of the log-concavity conjecture for genus distributions of graphs. For Conjecture \ref{conj:V-self}, the productions in Table 2.1 of \cite{Gr11a} lead to an expression for $\Gamma(X)$ as a linear combination of double-root partials and some of their offset sequences.  For Conjecture \ref{conj:E-self}, Theorems 2.6 and 2.7 of \cite{PoKhGr12} give productions for the two ways to self-amalgamate a doubly edge-rooted graph.

\begin{conj} \label{conj:V-self}
Let $(H,u,v)$ be a doubly vertex-rooted graph with 2-valent roots.  Then the genus distribution of the graph $X$ formed by amalgamating the vertex-roots $u$ and $v$ is log-concave.
\end{conj}

\begin{conj} \label{conj:E-self}
Let $(H,e,f)$ be a doubly edge-rooted graph with 2-valent roots. Then the genus distributions of both the graphs that can formed by amalgamating the edge-roots $e$ and $f$ are log-concave.
\end{conj}

To the best of our knowledge, no one has found a graph whose genus distribution or partial genus distribution has internal zeros.  We have not encountered any graph with a non-log-concave partial genus distribution.  Based on these observations, we pose a third conjecture.
\smallskip

\begin{conj} \label{conj:pgd-lc}
Partial genus distributions for singly or doubly vertex-rooted or edge-rooted graphs are log-concave and have no internal zeros.
\end{conj}

\bigskip
%%%%%%%%%%%%%%%%%%%%%%%%%%%%%%%%%%%%%%%%%%%%%
%%%%%%%%%%%%%%%%%%%%%%%%%%%%%%%%%%%%%%%%%%%%%


\begin{thebibliography}{99}

\bibitem%[ASW52]
{ASW52}
M. Aissen, I.J. Schoenberg, and A. Whitney,
On generating functions of totally positive \hbox{sequences I},
\textsl{J. Anal. Math.} \textbf{2} (1952), 93--103.


\bibitem
{BW09B}
L.W. Beineke, R.J. Wilson, J.L. Gross, and T.W. Tucker, editors,
\textsl{Topics in Topological Graph Theory},
Cambridge Univ.\ Press, 2009.

\bibitem%[BBL09]
{BBL09}
J. Borcea, P. Br\"and\'en, and T.M. Liggett,
Negative dependence and the geometry of polynomials,
\textsl{J. Amer. Math. Soc.} \textbf{22(2)} (2009), 521--567.

\bibitem%[Br89]
{Br89} F. Brenti,
	Unimodal, log-concave and P\'olya frequency sequences in combinatorics,
	\textsl{Mem. Amer. Math. Soc.} \textbf{413} (1989).

\bibitem%[Br94]
{Br94}
F. Brenti, Log-concave and unimodal sequences in algebra, combinatorics, and geometry: An update, \textsl{Contemp. Math.} \textbf{178} (1994), 71--89.

\bibitem%[But87]
{But87}
L.M. Butler,
A unimodality result in the enumeration
of subgroups of a finite abelian group,
\textsl{Proc. Amer. Math. Soc.} \textbf{101(4)} (1987), 771--775.

\bibitem%[CG09]
{CG09}
W.Y.C. Chen and C.C.Y. Gu,
	The reverse ultra log-concavity of the Boros-Moll polynomials,
	\textsl{Proc. Amer. Math. Soc.} \textbf{137(12)} (2009), 3991--3998.

\bibitem
{CPQ12}
W.Y.C. Chen, S.X.M. Pang, and E.X.Y. Qu,
Partially $2$-colored permutations and the Boros-Moll polynomials,
\textsl{Ramanujan J.} \textbf{27} (2012), 297--304.

\bibitem%[CTWY10]
{CTWY10}
W.Y.C. Chen, R.L. Tang, L.X.W. Wang, and A.L.B. Yang,
The $q$-log-convexity of the Narayana polynomials of type $B$,
\textsl{Adv. in Appl. Math.} \textbf{44} (2010), 85--110.

\bibitem%[CWY10]
{CWY10}
W.Y.C. Chen, L.X.W. Wang, and A.L.B. Yang,
Schur positivity and the $q$-log-convexity of the narayana polynomials,
\textsl{J. Algebraic Combin.} \textbf{32} (2010), 303--338.

\bibitem%[CX09]
{CX09}
W.Y.C. Chen and E.X.W. Xia,
The ratio monotonicity of the Boros-Moll polynomials,
\textsl{Math. Comp.} \textbf{78(268)} (2009), 2269--2282.

\bibitem %[CGM12a]
{ChGrMa12a} Y. Chen, J.L. Gross, and T. Mansour,
	Genus distributions of star-ladders,
	\textsl{Discrete Math.} \textbf{312} (2012), 3029--3067.
	
\bibitem %[CGM12b]
{ChGrMa12b} Y. Chen, J.L. Gross, and T. Mansour,
	Total embedding distributions of circular ladders,
	\textsl{J.\ Graph Theory} \textbf{74} (2013), 32--57.  Online 9 August 2012.
	
\bibitem %[CMZ11]
{ChMaZo11} Y. Chen, T. Mansour, and Q. Zou,
	Embedding distributions of generalized fan graphs,
	\textsl{Canad. Math. Bull.} \textbf{56} (2013), 265--271.  Online 31 August 2011.	
	
\bibitem% [CMZ12]
{ChMaZo12} Y. Chen, T. Mansour, and Q. Zou,
	Embedding distributions and Chebyshev polynomials,
	\textsl{Graphs Combin.} \textbf{28} (2012), 597--614.	
	
\bibitem{CST}
M.D.E.~Conder, J.~{\v S}ir{\' a}{\v n}, and T.W.~Tucker,
The genera, reflexibility and simplicity of regular maps,
\textsl{J.~Eur.~Math.~Soc.} \textbf{12} (2010), 343--364.

\bibitem{Duke66}
R.A. Duke, The genus, regional number, and Betti number of a graph,
\textsl{Canad. J. Math.} \textbf{18} (1966), 817--822.


\bibitem{Efr65}
B. Efron,
Increasing properties of P\'olya frequency functions,
\textsl{Ann. Math. Stat.} \textbf{36(1)}, (1965), 272--279.

\bibitem%[FGS89]
{FuGrSt89}  M. Furst, J.L. Gross, and R. Statman,
	Genus distributions for two class of graphs,
	\textsl{J. Combin. Theory Ser. B} \textbf{46} (1989), 523--534.

\bibitem%[Gr11a]
{Gr11a} J.L. Gross,
	Genus distribution of graph amalgamations: Self-pasting at root-vertices,
	{\sl Australas. J. Combin.} \textbf{49} (2011), 19--38.

\bibitem%[Gr11b]
{Gr11b}  J.L. Gross,
	Genus distributions of cubic outerplanar graphs,
  	 \textsl{J. Graph Algorithms Appl.} \textbf{15} (2011), 295--316.

\bibitem%[Gr12b]
{Gr12b}  J.L. Gross,
	Embeddings of graphs of fixed treewidth and bounded degree, \textsl{Ars Math. Contemp.} \textbf{7} (2014), 127--148.   Online December 2013.
	
\bibitem%[Gr13]
{Gr13}  J.L. Gross,
	Embeddings of cubic Halin graphs: a surface-by-surface inventory,
   	\textsl{Ars Math. Contemp.} \textbf{7} (2013),  37--56.  %Online June 2012.

\skipit{
\bibitem%[GrFu87]
{GrFu87}  J.L. Gross and M. Furst,
	Hierarchy for imbedding-distribution invariants of a graph,
	\textsl{J. Graph Theory} \textbf{11} (1987), 205--220.
\DWm{Uncited.}
} %% end-skipit
	
\bibitem%[GKP10]
{GKP10}  J.L. Gross, I.F. Khan, and M.I. Poshni,
	Genus distribution of graph amalgamations: Pasting at root-vertices,
	{\sl Ars Combin.} \textbf{94} (2010), 33--53.

\skipit{
\bibitem%[GrKo12]
{GrKo12}  J.L. Gross and M. Kotrb\v c\'ik,
	Genus distributions of cubic series-parallel graphs,
	preprint, 2013.
\DWm{Uncited.}
} %% end-skipit

\bibitem
{GMT14} J.L. Gross, T. Mansour, and T.W. Tucker,
	Log-concavity of genus distributions of ring-like families of graphs,
	\textsl{European J. Combin}, 20pp, to appear

\bibitem%[GRT89]
{GrRoTu89}  J.L. Gross, D.P. Robbins, and T.W. Tucker,
	Genus distributions for bouquets of circles,
	\textsl{J.\ Combin.\ Theory Ser. B} \textbf{47} (1989), 292--306.

\bibitem%[GrTu87]
{GrTu87} J.L. Gross and T.W. Tucker,
	{\sl Topological Graph Theory}, Dover, 2001 (original ed. Wiley, 1987).

\bibitem{Gro}  A. Grothendieck,
	{\sl ``Esquisse d'un programme''}, preprint, Montpellier, 1984.

\bibitem
{Hea1890} P.J. Heawood,
	Map-colour theorem, \textsl{Quart.\ J. Math.} \textbf{24} (1890), 332--338.

\bibitem
{Hef1891} L. Heffter,
	\"Uber das Problem der Nachbargebiete, \textsl{Math.\ Ann.} \textbf{38}, 477--508.

\bibitem{Her72}
A.P. Heron,
\textsl{Matroid polynomials}, Combinatorics
(Proc. Conf. Combinatorial Math., Math. Inst., Oxford, 1972),
Institute for Mathematics and its Applications,
Southend-on-Sea (1972), 164--202.

\bibitem{Huh12}
J. Huh,
Milnor numbers of projective hypersurfaces and the chromatic polynomial of graphs,
\hbox{\textsl{J. Amer. Math. Soc.}} \textbf{25(3)} (2012), 907--927.

\bibitem
{Jac87} D.M.\ Jackson,
Counting cycles in permutations by group characters, with an application to a topological problem,  \textsl{Trans.\ Amer.\ Math.\ Soc.} \textbf{299} (1987), 785--801.

\bibitem{JoPr83}
K. Joag-Dev and F. Proschan,
Negative association of random variables with applications,
\textsl{Ann. Statist.} \textbf{11(1)} (1983), 286--295.

\bibitem{JoSi}  G.A. Jones and D. Singerman,
Bely\u\i\/ functions, hypermaps, and Galois groups,
	\textsl{Bull. London Math. Soc.} \textbf{28} (1996), 561--590.

\bibitem%[Kar68B]
{Kar68B}
S. Karlin,
	{\sl Total Positivity}, Stanford Univ. Press, 1968.

\bibitem%[KP07]
{KP07}
M. Kauers and P. Paule,
A computer proof of Moll's log-concavity conjecture,
\textsl{Proc. Amer. Math. Soc.} \textbf{135(12)} (2007), 3847--3856.

\bibitem%[KPG12]
{KhPoGr12}  I.F. Khan, M.I. Poshni, and J.L. Gross,
	Genus distribution of $P_3\Box P_n$,
	\textsl{Discrete Math.} \textbf{312} (2012), 2863--2871.

\bibitem%[KiLe98]
{KiLe98} J.H. Kim and J. Lee,
	Genus distributions for bouquets of dipoles,
	\textsl{J. Korean Math. Soc.} \textbf{35} (1998), 225--234.

\bibitem{Ko36}
D. K\"onig,
\textsl{Theorie der endlichen und unendlichen Graphen},
Akademische Verlagsgesellschaft, 1936.

\bibitem%[Kra89]
{Kra89}
C. Krattenthaler,
On the $q$-log-concavity of Gaussian binomial coefficients,
\textsl{Monatsh. Math.} \textbf{107} (1989), 333--339.

\skipit{
\bibitem%[KwLe93]
{KwLe93} J.H. Kwak and J. Lee,
	Genus polynomials of dipoles,
	\textsl{Kyungpook Math. J.} \textbf{33} (1993), 115--125.
}

\bibitem%[Lig97]
{Lig97}
T.M. Liggett,
Ultra logconcave sequences and negative dependence,
\textsl{J. Combin. Theory Ser. A} \textbf{79} (1997), 315--325.
%In the published version of the above paper,
%the word ``logconcave'' was presented without the dash symbol.

\bibitem%[LW07]
{LW07} L.L. Liu and Y. Wang,
	A unified approach to polynomial sequences with only real zeros,
	\textsl{Adv. in Appl. Math.} \textbf{38} (2007), 542--560.

\bibitem%[MS10]
{MS10}
P.R.W. McNamara and B.E. Sagan,
Infinite log-concavity: Developments and conjectures,
\textsl{Adv. in Appl. Math.} \textbf{44} (2010), 1--15.

\bibitem%[Men69]
{Men69}
K.V. Menon,
On the convolution of logarithmically concave sequences,
\textsl{Proc. Amer. Math. Soc.} \textbf{23(2)} (1969), 439--441.

\bibitem
{Mo89} B. Mohar,
An obstruction to embedding graphs in surfaces, \textsl{Discrete Math.} \textbf{78} (1989), 135--142.

\bibitem%[Pem00]
{Pem00}
R. Pemantle,
Towards a theory of negative dependence,
\textsl{J. Math. Phys.} \textbf{41(3)} (2000), 1371--1390.

\bibitem%[Pit97]
{Pit97}
J. Pitman, Probabilistic bounds on the coefficients of polynomials with only real zeros,
\textsl{J. Combin. Theory Ser. A} \textbf{77} (1997), 279--303.

\bibitem%[PKG10]
{PoKhGr10}  M.I. Poshni, I.F. Khan, and J.L. Gross,
	Genus distribution of graphs under edge-amalgamations,
   	\textsl{Ars Math. Contemp.} \textbf{3} (2010), 69--86.

\bibitem%[PKG11]
{PoKhGr11}  M.I. Poshni, I.F. Khan, and J.L. Gross,
	Genus distribution of 4-regular outerplanar graphs,
	\textsl{Electron. J. Combin.} \textbf{18} (2011) \#P212, 25pp.
	
\bibitem%[PKG12]
{PoKhGr12}  M.I. Poshni, I.F. Khan, and J.L. Gross,
	Genus distribution of graphs under self-edge-amalgamations,
	\textsl{Ars Math. Contemp.} \textbf{5} (2012), 127--148.

\bibitem {Read68}
R.C. Read,
	An introduction to chromatic polynomials,
	\textsl{J. Combin.\ Theory} \textbf{4} (1968), 52--71.

\bibitem%[Rie90]
{Rie90} R.G. Rieper,
	The enumeration of graph embeddings,
	Ph.D. thesis, Western Michigan Univ., 1990.

\bibitem
{Ri74} G. Ringel,
	\textsl{Map Color Theorem}, Springer-Verlag, 1974.


\bibitem{RS90} N. Robertson and P. Seymour,
	Graph minors VIII. A Kuratowski theorem for general surfaces,
	{\sl J. Combin. Theory Ser. B} \textbf{48} (1990), 255--288.

	 	
\bibitem {RS2}	N. Robertson and P. Seymour,
	Graph minors XX. Wagner's Conjecture,
	\textsl{J.\ Combin.\ Theory Ser.~B}~\textbf{92} (2004), 325--357.

\bibitem{Rota71}
G.-C. Rota,
	Combinatorial theory, old and new,
	\textsl{Actes Congr. Internat. Math.} (Nice, 1970),
	Tome~\textbf{3}, Gauthier-Villars, Paris (1971), 229--233.

\bibitem{Sag88}
B.E. Sagan,
	Inductive and injective proofs of log concavity results,
	\textsl{Discrete Math.} \textbf{68} (1988), 281--292.

\bibitem%[Sag92]
{Sag92}
B.E. Sagan,
	Log concave sequences of symmetric functions
	and analogs of the Jacobi-Trudi determinants,
	\textsl{Trans. Amer. Math. Soc.} \textbf{329(2)} (1992), 795--811.

\bibitem%[Sch55]
{Sch55}
I.J. Schoenberg,
	On the zeros of the generating functions of multiply positive sequences and functions,
	\textsl{Ann. Math.} \textbf{62} (1955), 447--471.

\bibitem%[Stah97]
{Stah97} S. Stahl,
	On the zeros of some genus polynomials,
	\textsl{Canad. J. Math.} \textbf{49} (1997), 617--640.

\bibitem%[Stan89]
{Stan89} R.P. Stanley,
	Log-concave and unimodal sequences in algebra, combinatorics, and geometry, ¬†
	\textsl{Ann. New York Acad. Sci.} {\bf 576} (1989), 500--534.

\bibitem%[Stan00]
{Stan00}
R.P. Stanley,
	Positivity problems and conjectures
in algebraic combinatorics, Mathematics: frontiers and perspectives,
\textsl{Amer. Math. Soc.}, Providence, RI, 2000, 295--319.

\skipit{
\bibitem%[ViWi07]
{ViWi07} T.I. Visentin and S.W. Wieler,
	On the genus distribution of $(p,q,n)$-dipoles,
	\textsl{Electron. J. Combin.} \textbf{148} (2007) \#R12, 18pp.
}

\bibitem%[Wag97]
{Wag97} D.G. Wagner,
	Zeros of genus polynomials of graphs in some linear families,
	Univ. Waterloo Research Repot CORR 97-15 (1997), 9pp.

\bibitem%[WZ13]
{WZ13} D.G.L. Wang and T.Y. Zhao,
The real-rootedness and log-concavities of coordinator polynomials of Weyl group lattices,
\textsl{European J. Combin.} \textbf{34} (2013), 490--494.

\bibitem%[WZ08]
{WZ08}
J. Wang and H. Zhang,
$q$-Weighted log-concavity and $q$-product theorem on the normality of posets,
\textsl{Adv. in Appl. Math.} \textbf{41} (2008), 395--406.

\bibitem%[WY07]
{WY07}
Y. Wang and Y-N. Yeh,
Log-concavity and LC-positivity,
\textsl{J. Combin. Theory Ser. A} \textbf{114} (2007), 195--210.

\bibitem{Wel76}
D. Welsh,
\textsl{Matroid Theory},
London Math. Soc. Monogr. Ser. \textbf{8},
Academic Press, London-New York, 1976.


\end{thebibliography}
\end{document}